\documentclass[a4paper,10pt]{amsart}
\usepackage{amssymb,amsmath,amsfonts,amsthm}
\usepackage[dvips]{graphicx}
\usepackage{psfrag}

\usepackage{color}
\usepackage{todonotes}

\theoremstyle{plain}
\newtheorem{main}{Theorem}

\newtheorem{maincor}[main]{Corollary}
\newtheorem{theorem}{Theorem}[section]
\newtheorem{lemma}[theorem]{Lemma}
\newtheorem{proposition}[theorem]{Proposition}
\newtheorem{corollary}[theorem]{Corollary}
\theoremstyle{remark}
\newtheorem{remark}[theorem]{Remark}
\newtheorem{definition}[theorem]{Definition}

\newcommand{\quand}{\quad\text{and}\quad}
\newcommand{\Leb}{\operatorname{vol}}

\newcommand{\C}{\operatorname{C}}

\newcommand{\Gibbs}{\operatorname{Gibbs}}

\newcommand{\Jac}{\operatorname{Jac}}
\newcommand{\diam}{\operatorname{diam}}
\newcommand{\G}{\operatorname{G}}

           \def\ea{\end{array}}
          \def\ec{\end{center}}
     \def\ed{\end{description}}
        \def\ee{\end{equation}}
       \def\eea{\end{eqnarray}}
     \def\eeaa{\end{eqnarray*}}
 \def\et{\end{thebibliography}}
\def\bib{\bibitem}

       \def\nt{\noindent}

\def\Diff{{\rm Diff}}

\def\Gibb{{\rm Gibb}}
\def\inv{{\rm inv}}

\def\supp{\operatorname{supp}}

\def\cG{{\mathcal G}}
\def\cA{{\mathcal A}}

\def\cC{{\mathcal C}}

\def\cI{{\mathcal I}}

\def\cU{{\mathcal U}}
\def\cV{{\mathcal V}}
\def\cR{{\mathcal R}}
\def\cB{{\mathcal B}}

\def\cF{{\mathcal F}}
\def\cM{{\mathcal M}}

\def\cR{{\mathcal R}}

\def\loc{\operatorname{loc}}

\def\vep{\varepsilon}

\def\TT{{\mathbb T}}

\title{Entropy along expanding foliations}

\author{Jiagang Yang}
\date{\today}
\thanks{J.Y. is partially supported by CNPq, FAPERJ, and PRONEX.}

\address{Department of Mathematics, Southern University of Science and Technology of China, 1088 Xueyuan Rd., Xili, Nanshan District, Shenzhen, Guangdong, China 518055}
\address{Departamento de Geometria, Instituto de Matem\'atica e Estat\'istica, Universidade
Federal Fluminense, Niter\'oi, Brazil}
\email{yangjg\@@impa.br}

\begin{document}

\begin{abstract}
The (measure-theoretical) entropy of a diffeomorphism along an expanding invariant foliation
is the rate of complexity generated by the diffeomorphism along the leaves of the foliation.
We prove that this number varies upper semi-continuously with the diffeomorphism ($\C^1$
topology), the invariant measure (weak* topology) and the foliation itself in a suitable
sense.

This has several important consequences. For one thing, it implies that the set of Gibbs
$u$-states of $\C^{1+}$ partially hyperbolic diffeomorphisms is an upper semi-continuous function
of the map in the $\C^1$ topology. Another consequence is that the sets of partially hyperbolic
diffeomorphisms with mostly contracting or mostly expanding center are $\C^1$ open. New examples of
partially hyperbolic diffeomorphisms with mostly expanding center are provided, and the existence of
physical measures for $C^1$ residual subset of diffeomorphisms are discussed.

We also provide a new class of robustly transitive diffeomorphisms:
every $C^2$ volume preserving, accessible partially hyperbolic diffeomorphism with one dimensional center
and non-vanishing center exponent is $C^1$ robustly transitive (among neighborhood of diffeomorphisms
which are not necessarily volume preserving).
\end{abstract}

\maketitle

\setcounter{tocdepth}{1} \tableofcontents

\section{Introduction}

A continuous foliation $\cF$ with smooth leaves is \emph{expanding for a diffeomorphism $f$}
if it is invariant under the iteration of $f$ and the derivative $Df$ restricted to the tangent bundle of
$\cF$ is uniformly expanding.

The simplest example of expanding foliation should be the unstable foliation.
As shown by Anosov \cite{A}, those foliations, while not smooth,
are still regular enough (absolute continuity).
The recent notable results on ergodicity of volume preserving partially hyperbolic,
see \cite{GPS,PSh97,PSh00,BW,RRU}, also depends on the analysis of the stable and unstable foliations.
More precisely, the julienne quasi-conformality property.

In general, given an expanding foliation, its tangent bundle may not correspond to the
strongest expansion, and an invariant transverse sub-bundle need not exist either. For instance,one can take the  weak-expanding foliation of a linear Anosov diffeomorphism.
Moreover, general expanding foliations are not
necessarily absolutely continuous (see for instance \cite{VY3}).

In this paper we will build the entropy theory on expanding foliations. The \emph{partial entropy} of an invariant
probability measure along an expanding foliation, whose precise definition will be given in Section~\ref{ss.partialentropy} (see also~\cite{LY85a, VY3}), is a value that quantifies the complexity of the measure generated
on this foliation.

Let $\mu$
be any invariant measure of $f$ and denote the \emph{partial entropy of $\mu$ along the foliation
$\cF$} by $h_\mu(f,\cF)$. Our first theorem shows that the partial entropy is upper semi-continuous:



\begin{main}\label{main.D}

Let $f_n$ be a sequence of $\C^1$ diffeomorphisms which converge to $f$ in the $\C^1$ topology,
and $\mu_n$ invariant measures of $f_n$ which converge to an invariant measure $\mu$ of $f$ in the
weak* topology. Suppose $\cF_n$ is an expanding foliation of $f_n$ for each $n$, with $\cF_n\to \cF$
(in the sense of Definition~\ref{d.convergenceoffoliaiton}), then
$$\limsup h_{\mu_n}(f_n,\cF^u_n)\leq h_\mu(f,\cF^u_f).$$
\end{main}

The research on the regularity of entropy has a long history, one can find more references from
\cite{Yom,New,LVY,VY2}. Our proof of Theorem~\ref{main.D} is inspired by the dimension theory of invariant measures (see
\cite{Y82,LY85b,BPS}), and the Pesin entropy formula (\cite{L84,LS82,LY85a}).
Similar result with Theorem \ref{main.D} can be found in \cite{HHW}.

The regularity of partial entropy has several important consequences, as we will explain in the following sections.

Several other applications of the present methods have been found in the meantime, some of which had not been foreseen. In a joint paper with
Tahzibi \cite{TY}, the regularity of partial entropy combined with other techniques is used to handle measures of partially hyperbolic
systems with large entropy. This is also a main ingredient in the joint papers with Liang, Marin \cite{LMY1,LMY2} and with Viana \cite{VY4},
where regularity of Lyapunov exponents of partially hyperbolic systems are analyzed. The notion of partial entropy for
expanding foliation is also used in the joint paper with Viana \cite{VY3} on center foliations, and with Saghin \cite{SY} on
the rigidity of Anosov systems.

\subsection{Partially hyperbolic diffeomorphisms and Gibbs $u$-states}

\emph{Partially hyperbolic} diffeomorphisms were proposed by Brin, Pesin~\cite{BP} and Pugh,
Shub~\cite{PS72} independently at the early 1970's, as an extension of the class of Anosov
diffeomorphisms \cite{A,AS}. A diffeomorphism $f$ is said to be  partially hyperbolic if there exists a
decomposition $TM = E^s \oplus E^c \oplus E^u$ of the tangent bundle $TM$ into three
continuous invariant sub-bundles $E^s_x$ and $E^c_x$ and $E^u_x$, such that $Df \mid E^s$ is
uniform contraction, $Df\mid E^u$ is uniform expansion and $Df \mid E^c$ lies in between:
$$
\frac{\|Df(x)v^s\|}{\|Df(x)v^c\|} \le \frac 12
\quand
\frac{\|Df(x)v^c\|}{\|Df(x)v^u\|} \le \frac 12
$$
for any unit vectors $v^s\in E^s$, $v^c\in E^c$, $v^u\in E^u$ and any $x\in M$.

Partially hyperbolic diffeomorphisms form an open subset of the space of $C^r$ diffeomorphisms
of $M$, for any $r\ge 1$. The \emph{stable sub-bundle} $E^s$ and the \emph{unstable sub-bundle}
$E^u$ are uniquely integrable, that is, there are unique foliations: the stable foliation $\cF^s$ and
the unstable foliation $\cF^u$,
whose leaves are smooth immersed sub-manifolds of $M$ tangent to $E^s$ and $E^u$, respectively,
at every point. The unstable and stable foliations are expanding foliations for the partially hyperbolic
diffeomorphism and its inverse respectively.

Following Pesin and Sinai~\cite{PS82} and Bonatti and Viana~\cite{BV00} (see also~\cite[Chapter 11]{BDVnonuni}),
we call \emph{Gibbs $u$-state} any invariant probability measure whose conditional probabilities
(Rokhlin~\cite{Rok49}) along strong unstable leaves are absolutely continuous with respect to the
Lebesgue measure on the leaves. In fact, assuming the derivative $Df$ is H\"older continuous, the
Gibbs-$u$ state always exists; furthermore, the densities with respect to Lebesgue measures along unstable
leaves are continuous due to distortion. Moreover, making use of uniform distortion, the densities vary continuously
with respect to the strong unstable leaves, and  and the diffeomorphisms under of $\C^{1+\vep}$ topology. As a consequence,
the space of Gibbs $u$-states, denoted by $\Gibb^u(\cdot)$, is compact relative to the weak-* topology
in the probability space, and varies upper semi-continuously with respect to the diffeomorphism in $\C^{1+\vep}$ topology
(\cite[Remark 11.15]{BDVnonuni}). In this article, we build a similar result in the $\C^1$ topology:

\begin{main}\label{main.A}

$\Gibb^u(\cdot)$ varies upper semi-continuously among the $\C^{1+}$ partially hyperbolic
diffeomorphisms in the $\C^1$ topology.

\end{main}

The difficulty of the proof comes to the fact that uniform distortion in general fails under $C^1$ topology, that is, for
two $C^{1+\varepsilon}$ diffeomorphisms $C^1$ close to each other, the H\"{o}lder norms of the tangent maps can be vastly different.

Let us explain briefly how we avoid using the uniform $C^{1+\alpha}$ topology, and how the problem is associated with the regularity of
partial entropy. We use an equivalent definition of Gibbs $u$-states given by \cite{Y82}--one that involves the Pesin formula for partial entropy. To be more precise, a measures is a Gibbs $u$-state if its partial entropy along the unstable foliation coincides with the sum of exponents along the unstable bundle. By
the upper semi-continuity of partial entropy,  the limit of Gibbs $u$-states must satisfy the same equality, and hence is a Gibbs $u$-state.

\subsection{Physical measures}
Let $f:M\to M$ be a diffeomorphism on some compact Riemannian manifold $M$.
An invariant  probability $\mu$ is a \emph{physical measure} for $f$ if the
set of  points $z\in M$ for which
\begin{equation}\label{eq.SRBmeasure}
\frac 1n \sum_{j=0}^{n-1} \delta_{f^i(z)} \to \mu
\quad \text{(in the weak$^*$ sense)}
\end{equation}
has positive volume. This set is denoted $B(\mu)$ and called the \emph{basin}
of $\mu$. A program for investigating the physical measures of partially hyperbolic
diffeomorphisms was initiated by Alves, Bonatti, Viana in \cite{ABV00,BV00},
who proved the existence and finiteness when $f$ is either ``mostly expanding''
(asymptotic forward expansion) or ``mostly contracting'' (asymptotic forward
contraction) along the center direction.

The set of Gibbs $u$-states plays important roles in the study of physical measures of
partially hyperbolic diffeomorphisms. The partially hyperbolic diffeomorphisms with
\emph{mostly contracting center}, first studied
in~\cite{BV00}, are those $\C^{1+}$ partially hyperbolic diffeomorphisms where all the Gibbs $u$-states have negative center Lyapunov exponents. As a corollary of the semi-continuation of the set of Gibbs $u$-states
in the $\C^{1+\vep}$ topology, the set of partially hyperbolic diffeomorphisms
forms a $\C^{1+\vep}$ open set (see~\cite{BDPP},~\cite{And10},~\cite{VY1},~\cite{DVY}).

The notion of partially hyperbolic diffeomorphisms with mostly expanding center was
provided by Alves, Bonatti and Viana (\cite{ABV00}). More recently, Andersson and V\'asquez
proved in \cite{AC} that a partially hyperbolic diffeomorphism such that every Gibbs $u$-state has positive center exponents has mostly expanding center. They also proposed to the latter, somewhat stronger,
property as the actual definition of having mostly expanding center. We will use this definition in the present paper, and  prove that the set of diffeomorphisms satisfying this
condition is open; this is not true for the original definition in \cite{ABV00}, as observed in
\cite[Proposition A]{AC}.

As a corollary of Theorem~\ref{main.A}, we obtain the $\C^1$ openness of the partially hyperbolic
diffeomorphisms with mostly contracting center or with mostly expanding center, thus give a positive answer to
a question of Dolgopyat \cite{D}.

\begin{main}\label{main.B}

The sets of partially hyperbolic diffeomorphisms with mostly contracting center or mostly expanding
center are $\C^1$ open, that is, every $\C^{1+}$ partially hyperbolic diffeomorphism with
mostly contracting (resp. expanding) center admits a $\C^1$ open neighborhood, such that
every $\C^{1+}$ diffeomorphism in this neighborhood has also mostly contracting (resp. expanding)
center.

\end{main}

The only known example of diffeomorphism with mostly expanding center (in the stronger sense, as explained above) is due to Ma{\~{n}}{¡®{e}}~\cite{Ma}, see~\cite{ABV00} and~\cite[Section 6]{AC}.
As an application of Theorem~\ref{main.B}, we provide a whole new class of examples:

\begin{main}\label{main.C}

Let $f$ be a $C^{1+}$ volume preserving partially hyperbolic diffeomorphism
with one-dimensional center. Suppose the center exponent of the volume measure is
positive and $f$ is accessible, then $f$ admits an $\C^1$ open neighborhood (among diffeomorphisms
which are not necessarily volume preserving), such that every
$\C^{1+}$ diffeomorphism in this neighborhood has mostly expanding center
and admits a unique physical measure, whose basin has full volume.

\end{main}

Theorem~\ref{main.C} contains abundance of systems: by Avila~\cite{AA}, $\C^\infty$
volume preserving diffeomorphisms are $\C^1$ dense, and by Baraviera and Bonatti~\cite{BB},
the volume preserving partially hyperbolic diffeomorphisms with one-dimensional center
and non-vanishing center exponent are $\C^1$ open and dense. Moreover, the subset
of accessible systems is $\C^1$ open and $\C^k$ dense for any $k\geq 1$ among all
partially hyperbolic diffeomorphisms with one-dimensional center direction (Burns,
Rodriguez Hertz, Rodriguez Hertz, Talitskaya and Ures~\cite{BHHTU}; see also Theorem
1.5 in Ni{\c{t}}ic{\u{a}} and T{\"o}r{\"o}k~\cite{NT}).

One can also show that in the neighborhood described in Theorem~\ref{main.C}, typical $C^1$ diffeomorphisms also
admit physical measures:

\begin{maincor}\label{maincor.generic}

Let $f$ be a $C^{1+}$ volume preserving partially hyperbolic diffeomorphism
with one-dimensional center. Suppose the center exponent of the volume measure is
positive and $f$ is accessible. Then $f$ admits an $\C^1$ open neighborhood $\cU$ (among diffeomorphisms
which are not necessarily volume preserving),
such that every $C^1$ diffeomorphism $g$ in a residual subset $\cR$ of $\cU$ has a unique physical measure, whose basin has full volume.

\end{maincor}

The existence of physical measures for $C^1$ residual subset of maps were mainly studied for uniformly circle expanding maps
in \cite{CQ}, and for hyperbolic attractor in \cite{Q}. Note that in Corollary~\ref{maincor.generic}, no uniform hyperbolicity is assumed.

In the proof of Corollary~\ref{maincor.generic}, we make use of new classification/candidates of physical measures for $C^1$
partially hyperbolic diffeomorphisms (\cite{CYZ,HYY}), namely the set of invariant probabilities that satisfy the Pesin formula involving partial entropy. By the semi-continuity of partial entropy, we show that the class
of measures varies upper semi-continuously with respect to the maps; as a result, if one denote by $\cG$ the map that sends a diffeomorphism to the set of  candidate measures, then the continuous point of $\cG$ contains a $C^1$ residual subset $\cR$ of $\cU$.
Moreover, we show that for $C^{1+}$
diffeomorphisms in this neighborhood, the image under $\cG$ consists a unique element, and thus, by continuity, every $g\in \cU$ has a unique physical measure.

\subsection{Robustly transitive diffeomorphisms}
The first example of $C^1$-persistent,
non-hyperbolic, transitive (i.e. having a dense orbit in the whole manifold) diffeomorphisms was described
by M. Shub on the 4-dimensional torus $\TT^4=\TT^2 \times \TT^2$ as a skew product of an Anosov diffeomorphism
by ones that are  derived from an Anosov. Later, R. Ma\~n\'e \cite{Ma} proved that some derived-from-Anosov diffeomorphisms
on $\TT^3$ are $C^1$-persistently transitive (for a different proof, see \cite{PS}). Both examples depend on the manifold and topological structure of the maps.

In the remarkable paper \cite{BD}, Bonatti and D\'{i}az make use of 'blender'--a semi-local model, to provide
a much broader class of nonhyperbolic and transitive maps. In their construction, some perturbation (to obtain
blenders) is necessary, and one also need some global assumption on the invariant manifolds.
More precisely, the blender is a robust structure that can induce expanding property of the center bundle on some local unstable disks;
in order to extend the expansion from the blender to every unstable leaves, in general one need some extra hypothesis on the invariant manifold to guarantee that the iteration of invariant leaves will cross the blender .

In this paper, we provide a new measure-theoretical way to prove transitivity without using perturbation:

\begin{main}\label{main.transitive}
Let $f$ be a $C^2$ volume preserving, partially hyperbolic diffeomorphism with one dimensional center.
Suppose $f$ is accessible and the center exponent is not vanishing, then $f$ is $C^1$ robustly transitive, i.e.,
every $C^1$ diffeomorphism $g$ (not necessarily volume preserving) in a $C^1$ neighborhood of $f$ is transitive.
\end{main}

To show the transitively of $g$, we consider a special class of invariant measures $\G(\cdot)$ (see Definition~\ref{df.G}).
For a diffeomorphism $f$, the measures in $\G(f)$ have positive center exponent means that we have expansion along the center direction;
thus for point in a full volume subset, any of its neighborhood has uniform size
along the center-unstable direction after some iteration. By the analysis of $\G(f^{-1})$ (Lemma~\ref{l.Ginverse}), one
can show that for Lebesgue
almost every point, its negative orbit is almost dense, this implies the transitivity. Moreover, since $\G(\cdot)$ is defined using partial entropy, the
regularity of partial entropy enables us to prove that the two phenomenons above are robust.

Let us observe that although for volume preserving partially hyperbolic diffeomorphism, accessibility implies transitivity,
this is not true for dissipative partially hyperbolic diffeomorphisms, see the recent work of Y. Shi \cite{Sh}.
Y. Shi's work also shows that the assumption of non-vanishing exponent is necessary. One may further ask if every volume preserving, partially
hyperbolic diffeomorphism with one dimensional center and vanishing center exponent is always not robustly transitive.

\nt{\bf Outline of the work}.
In section 2 we give the necessary background which will be used throughout
the text. In Section 3 we build a sequence of measurable
partitions, which is used in Section 4 to prove Theorem \ref{main.D}. Section 5 is devoted to
the proofs of Theorems \ref{main.A} and \ref{main.B}. The proof of Theorem \ref{main.C} is divided into
Sections 6 and 7, which also contains the proof of  Corollary~\ref{maincor.generic}. And we give the proof of
Theorem~\ref{main.transitive} in Section~\ref{s.transitive}.

\section{Preliminary}

Throughout this section, let $f$ be a diffeomorphism on the manifold $M$, and $\mu$ be an invariant probability
measure of $f$.
\subsection{Volume preserving partially hyperbolic diffeomorphism}
We say a partially hyperbolic diffeomorphism is \emph{accessible} if any two points can be
joined by a piecewise smooth curve such that each leg is tangent to either $E^u$ or $E^s$ at every point.

Pugh, Shub conjectured in \cite{PSh97} that (essential) accessibility implies ergodicity, for a $\C^2$
partially hyperbolic, volume preserving diffeomorphism. In \cite{PSh00} they showed that this does hold
under a few additional assumptions. The following result is a special case of a general
result of Burns, Wilkinson~\cite{BW}:

\begin{proposition}\label{p.volumeergodic}

Every $\C^{1+\vep}$ volume preserving, accessible partially hyperbolic diffeomorphism with one-dimensional
center is ergodic.

\end{proposition}

\subsection{Continuation of foliation}

In this subsection we explain the convergence between foliations that is used in
Theorem~\ref{main.D}.

Let $\cF$ be a foliation of $M$ with dimension $l$, that is,
every leaf is a $l$-dimensional smooth immersed submanifold. An \emph{$\cF$-foliation
box} is the image $B$ of a topological embedding $\Phi : D^{d-l}\times D^{l}\to M$
such that every plaque $P_x = \Phi(\{x\}\times D^l)$ is contained in a leaf of $\cF$,
and every
$$\Phi(x,\cdot) : D^l \to M, y \mapsto \Phi(x,y)$$
is a $\C^1$ embedding depending continuously on $x$ in the $\C^1$ topology.
We write $D=\Phi(D^{d-l}\times \{0\})$, and denote this foliation
box by ($B, \Phi, D$).

Take a finite cover of $M$ consists of $\cF$-foliation boxes $\{(B_i,\Phi_i,D_i)\}_{i=1}^k$.
\begin{definition}\label{d.convergenceoffoliaiton}

We say \emph{a sequence of $l$-dimensional foliations $\cF_n$ converge to $\cF$} if:
\begin{itemize}
\item for each $n$, there exists a finite cover of $M$ by $\cF_n$-foliation boxes
      $$\{B^n_i,\Phi^n_i,D_i\}_{i=1}^k;$$
\item for each $1\leq i \leq k$, the topological embeddings $\Phi^n_i: D^{d-l}\times D^{l}\to M$
      converge uniformly to $\Phi_i$ in the $\C^0$ topology;
\item for every $x\in D_i$ ($1\leq i\leq k$), $\Phi^n(x, \cdot) : D^l \to M$ defined by
      $$y \mapsto \Phi^n(x,y)$$
      is a $\C^1$ embedding which converges to $\Phi(x,.)$ in the $\C^1$ topology as $n\to \infty$.
\end{itemize}
\end{definition}

\subsection{Measurable partitions and mean conditional entropy}
Let $\cB$ be the Borel $\sigma$-algebra on $M$ and $\mu$ a probability of $M$. In this subsection, we recall
the properties of measurable partitions, for more details see~\cite{Rok49, Rok67}.

\begin{definition}\label{d.measurablepartition}

A partition $\xi$ of $M$ is called \emph{measurable} if there is a sequence of finite partitions
$\xi_n$ $_{n\in \mathbb{N}}$ such that:
\begin{itemize}
\item elements of $\xi_n$ are measurable (up to $\mu$-measure $0$);

\item $\xi=\vee_{n}\xi_n$, that is, $\xi$ is the coarsest partition which refines $\xi_n$ for each $n$.

\end{itemize}
\end{definition}

For a partition $\xi$ and $x\in M$, we denote by $\xi(x)$ the element of $\xi$ which contains $x$.
For any measurable partition, we may define conditional measures of $\mu$ on almost every element.

\begin{proposition}\label{p.conditionalmeasure}
Let $\xi$ be a measurable partition. Then there is a full $\mu$-measure subset $\Gamma$
such that for every $x\in \Gamma$, there is a probability measure $\mu_x^{\xi}$ defined
on $\xi(x)$ satisfying:
\begin{itemize}
\item Let $\cB_\xi$ be the sub-$\sigma$-algebra of $\cB$ which consist unions of
      elements of $\xi$, then for any measurable set $A$, the function $x\to \mu^\xi_x(A)$ is
      $\cB_\xi$-measurable.

\item Moreover, we have
\begin{equation}\label{eq.conditional}
\mu(A)=\int \mu^\xi_x(A)d\mu(x).
\end{equation}
\end{itemize}
\end{proposition}
\begin{remark}
Let $\pi_\xi$ be the projection $M\to M/\xi$, and $\mu_\xi$ be the projection of measure $\mu$ onto
$M/\xi$ by the map $\pi_\xi$. Then equation~\eqref{eq.conditional} can be written as:
\begin{equation}\label{eq.projective}
\mu(A)=\int \mu^\xi_B(A)d\mu_\xi(B)
\end{equation}
where $B$ denotes the element of $\xi$ and $\mu^\xi_B$ the conditional measure on $B$.
\end{remark}

Let $\xi$ be a measurable partition and $C_1,C_2,\dots$ be the elements of $\xi$ of positive measure.
We define the \emph{entropy of the partition} by
\begin{equation}
H_\mu(\xi) = \begin{cases} \sum_k \phi(\mu(C_k)), & \mbox{if } \mu(M\setminus \cup_k C_k)=0 \\
\infty, & \mbox{if } \mu(M\setminus \cup_k C_k)>0 \end{cases}
\end{equation}
where $\phi: \mathbb{R}^+\to \mathbb{R}$ is defined by $\phi(x)=-x\log x$.

If $\xi$ and $\eta$ are two measurable partitions, then for every element $B$ of $\eta$, $\xi$
induces a partition $\xi_B$ on $B$. We define the \emph{mean conditional
entropy of $\xi$ respect to $\eta$}, denoted by $H(\xi\mid\eta)$, as the following:
\begin{equation}\label{eq.conditionalentropy}
H_\mu(\xi\mid\eta)=\int_{M/\eta} H_{\mu^\eta_B}(\xi_B) d\mu_\eta(B).
\end{equation}
\begin{definition}
For measurable partitions $\{\zeta_{n}\}_{n=1}^\infty$ and $\zeta$, we write $\zeta_n\nearrow \zeta$ if
the following conditions are satisfied:
\begin{itemize}
\item $\zeta_1<\zeta_2<\dots$;
\item $\vee_{n=1}^\infty \zeta_n=\zeta$.
\end{itemize}
\end{definition}

\begin{lemma}\label{l.increasingpartition}[\cite[Subsection 5.11]{Rok67}]
Suppose $\{\eta_n\}_{n=1}^\infty$, $\eta$ and $\xi$ are measurable partitions, such that $\eta_n\nearrow \eta$
and $H_\mu(\xi\mid \eta_1)<\infty$, then
$$H_\mu(\xi\mid \eta_n)\searrow H_\mu(\xi \mid \eta).$$
\end{lemma}

\begin{definition}
Let $\xi$ be a measurable partition, we put
$$h_\mu(f,\xi)=H_\mu(\xi\mid f\xi^+),$$
where $\xi^+=\vee_{n=0}^\infty f^n \xi$.
\end{definition}

\begin{remark}\label{r.increasingpartition}
A measurable partition $\xi$ is said to be \emph{increasing} if $f\xi <\xi$. For
an increasing partition $\xi$,
$$h_\mu(f,\xi)=H_\mu(\xi\mid f\xi).$$
\end{remark}

\subsection{Expanding foliations\label{ss.partialentropy}}

Throughout this subsection, $\cF$ denotes an expanding foliation of $f$. We are going to give
the precise definition of the partially metric entropy of $\mu$ along the expanding foliation
$\cF$, which depends on a special class of measurable partitions:

\begin{definition}
We say a measurable partition $\xi$ of $M$ is \emph{$\mu$-subordinate} to the $\cF$-foliation if
for $\mu$-a.e. $x$, we have
\begin{itemize}
\item[(1)] $\xi(x)\subset \cF(x)$ and $\xi(x)$ has uniformly small diameter inside $\cF(x)$;
\item[(2)] $\xi(x)$ contains an open neighborhood of $x$ inside the leaf $\cF(x)$;
\item[(3)] $\xi$ is an increasing partition, meaning that $f\xi \prec \xi$.
\end{itemize}

\end{definition}

Ledrappier, Strelcyn~\cite{LS82} proved that the Pesin unstable lamination admits some $\mu$-subordinate
measurable partition, which can also be applied on a general expanding foliation. In the proof of Theorem~\ref{main.D} we will need a uniform construction of these partitions for a sequence of diffeomorphisms and measures, which will be provided in Section~\ref{s.construction}.

The following result (for the subordinate partitions constructed as in Section~\ref{s.construction})
is contained in Lemma~3.1.2 of Ledrappier, Young~\cite{LY85a}:

\begin{lemma}\label{l.definitionleafentropy}
Given any expanding foliation $\cF$, we have $h_\mu(f,\xi_1)=h_\mu(f,\xi_2)$
for any measurable partitions $\xi_1$ and $\xi_2$ that are $\mu$-subordinate to $\cF$.
\end{lemma}

This allows us to give the following definition:

\begin{definition}\label{d.partialentropy}

The \emph{partial $\mu$-entropy} $h_\mu(f,\cF)$ of the expanding foliation $\cF$
is defined by $h_\mu(f,\xi)$ for any $\mu$-subordinate partition constructed as
in Section~\ref{s.construction}.
\end{definition}

\section{Construction of subordinate measurable partitions\label{s.construction}}

Let $f_n$ be a sequence of diffeomorphisms which converge to $f_0$ in the $\C^1$ topology, and
$\cF_n$ an expanding foliation of $f_n$ such that $\cF_n$ converge to $\cF$.
And $\{(B^n_i,\Phi^n_i,D_i)\}_{i=1}^k$ and $\{(B_i,\Phi_i,D_i)\}_{i=1}^k$ are the foliation boxes
of $\cF_n$ and $\cF$ respectively as in the Definition~\ref{d.convergenceoffoliaiton}. For simplicity,
we assume that each plaque of every foliation box has diameter bounded by one.

The main goal for this section is to construct measurable partition that is $\mu_n$-subordinate to expanding foliation
$\cF_n$ for each $n$ in a uniform way, which is done in Lemma~\ref{l.singlesubordinate}. The construction
can be divided into two steps:
the first step is to choose a finite partition $\cA$ of $M$ such that
\begin{itemize}
\item every element of $\cA$ is contained in some foliation
chart,
\item the neighborhood of $\partial \cA$ has small measure for $\mu$ and for every $\mu_n$ where $n\geq 1$
      (see \eqref{eq.smallboundary}).
\end{itemize}
Let $\cA^{\cF}$ (resp. $\cA^{\cF_n})$) be the partition such that every element is the intersection
between an element of $\cA$ and a local $\cF$ (resp. $\cF_n$) plaque of the corresponding foliation box.
Then the second step is to show that $\vee_{i=0}^\infty f^i(\cA^{\cF})$ (resp.
$\vee_{i=0}^\infty f_n^i(\cA^{\cF_n})$) is subordinate to $\cF$ (resp. $\cF_n$).

Take $r_0\ll 1$ a Lebesgue number of the open covering $\{B_i\}_{i=1}^k$, that is, there is a function
\begin{equation}\label{eq.definitionofI}
I: M\to \{1,\dots, k\} \text{ such that } B_{r_0}(x)\subset B_{I(x)}.
\end{equation}
When $n$ is sufficiently large, by the
definition of convergency of foliations, we have
\begin{equation}\label{eq.uniformLebesguenumber}
B_{r_0}(x)\subset B^n_{I(x)}.
\end{equation}
After removing a finite sequence, we assume \eqref{eq.uniformLebesguenumber} holds for all $n\geq 1$.

We need the following proposition whose proof we postpone to Appendix~\ref{ss.smallboundary}

\begin{proposition}\label{p.smallboundary}
Let $\{\nu_n\}_{n=0}^\infty$ be a sequence of probability measures on $M$. Then for any
$0<\lambda<\lambda^{\prime}<1$ and $R>0$, there is a finite partition $\cA$ of $M$ and
$C_n>0$ for every $n\in \mathbb{N}$, such that
\begin{itemize}
\item the diameter of every element of $\cA$ is less than $R$,
\item $\nu_n(B_{\lambda^i}(\partial\cA))\leq C_n(\lambda^{\prime})^i$, for every $n,i\in\mathbb{N}$,
      where $B_r(\partial \cA)$ denotes the $r$ neighborhood of $\partial \cA$.
\end{itemize}
\end{proposition}

Take $a>1$ such that for any $x\in M$ and $n\geq 1$,
\begin{equation}\label{eq.uniformyexpanding}
\|Df_n|_{T_x\cF_n(x)}\|>a.
\end{equation}
Applying Proposition~\ref{p.smallboundary} for
\begin{itemize}
\item $\nu_n=\mu_n$ for any $n\geq 0$, where we write $\mu_0=\mu$;
\item $R=r_0$;
\item $\lambda=\frac{1}{a}$ and $\lambda^\prime$,
\end{itemize}
we obtain a partition $\cA$ and $C_n>0$ ($n\in \mathbb{N}$) such that $\diam(\cA)<r_0$ and
\begin{equation}\label{eq.smallboundary}
\mu_n(B_{\lambda^i}(\partial\cA))\leq C_n(\lambda^{\prime})^i,\text{ for } n,i\geq 0.
\end{equation}

Recall that every element of the partition $\cA^{\cF}$ (resp. $\cA^{\cF_n}$) is the intersection
between $\cA$ and a local $\cF$ (resp. $\cF_n$) plaque in the corresponding foliation box.

\begin{lemma}\label{l.singlesubordinate}
$\vee_{i=0}^\infty f^i(\cA^{\cF})$ is subordinate to $\cF$ and $\vee_{i=0}^\infty f_n^i(\cA^{\cF_n})$ is subordinate to $\cF_n$ for every $n>0$.
\end{lemma}

\begin{proof}[Proof of Lemma\ref{l.singlesubordinate}:]
We only prove the first part of this lemma, the proof of the second part is similar.

Because $\mu$ is $f$ invariant, by \eqref{eq.smallboundary}, we have that
$$\sum_{j=1}^\infty \mu(f^j(B_{(\frac{1}{a})^j}(\partial\cA)))=\sum_{j=1}^\infty \mu(B_{(\frac{1}{a})^j}(\partial\cA))<\infty.$$
Hence, there is a $\mu$ full measure subset $Z$ and a function $\cI: Z \to \mathbb{N}$, such
that for every $x\in Z$ and any $j>\cI(x)$, $x\notin f^j(B_{(\frac{1}{a})^j}(\partial\cA))$,
or equivalently,
\begin{equation}\label{eq.awayboundary}
f^{-j}(x)\notin B_{(\frac{1}{a})^j}(\partial\cA).
\end{equation}

Because $\mu(\partial{\cA})=\mu(\bigcap B_{(\frac{1}{a})^j}(\partial\cA))=0$, after
removing a zero measure subset, we can assume that for every $x\in Z$ and any
$j\in \mathbb{Z}$, $f^j(x)\notin \partial{\cA}$. This hypothesis implies that for
every $m>0$, $\vee_{j=0}^{m} f^j(\cA^{\cF})(x)$ contains an open neighborhood of
$x$ inside the leaf $\cF(x)$.

Then this lemma follows from the claim that
$\vee_{j=0}^m f^j(\cA^{\cF})(x)=\vee_{j=0}^{\cI(x)} f^j(\cA^{\cF})(x)$ for any $m\geq \cI(x)$,
since this implies that $\vee_{j=0}^\infty f^j(\cA^{\cF})(x)=\vee_{j=0}^{\cI(x)} f^j(\cA^{\cF})(x)$.

To prove this claim, we only need observe that every plaque of each foliation box
($B_i,\Phi_i,D$) has diameter bounded by 1. Suppose by contradiction that there is $m\geq \cI(x)$,
such that $\vee_{j=0}^m f^j(\cA^{\cF})(x)\neq \vee_{j=0}^{m+1} f^j(\cA^{\cF})(x)$. This implies that
$f^{m+1}(\partial{\cA})\cap \vee_{j=0}^m f^j(\cA^{\cF})(x)\neq \emptyset$, i.e.,
$$d^{\cF}(f^{m+1}(\partial{\cA}),x)\leq 1,$$
where $d^{\cF}$ denotes the distance inside a leaf of the foliation $\cF$.
Then
$$d(f^{-(m+1)}(x),\partial{\cA})\leq d^{\cF}(f^{-(m+1)}(x),\partial{\cA})\leq (\frac{1}{a})^{m+1},$$
which contradicts \eqref{eq.awayboundary}, i.e., $f^{-(m+1)}(x)\notin B_{(\frac{1}{a})^{m+1}}(\partial\cA)$.

We conclude the proof of this Lemma, and hence, complete the construction.
\end{proof}

From the construction, it is easy to show that:

\begin{lemma}\label{l.boundedconditionalentropy}
$H_{\mu_n}(\cA^{\cF_n}\mid f_n(\cA^{\cF_n}))<\infty$.
\end{lemma}
\begin{proof}
By the previous construction, the diameter of every element of $\cA$ is bounded by $r_0$, which is
sufficiently small. Then every element $B$ of the partition $f_n(\cA^{\cF_n})$ is contained in a plaque of some
foliation box. Moreover, the partition $(\cA^{\cF_n})_B$ of $B$ induced by $\cA^{\cF_n}$ coincides
to the partition of $B$ induced by $\cA$, which is uniform finite. Because the metric entropy of
a partition is bounded by the logarithm of the number of its components, by \eqref{eq.conditionalentropy},
we have that
$$H_{\mu_n}(\cA^{\cF_n}|f_n(\cA^{\cF_n}))=\int_{M/f_n(\cA^{\cF_n})} H_{(\mu_n)_B^{f_n(\cA^{\cF_n})}}((\cA^{\cF_n})_B) d(\mu_n)_{f_n(\cA^{\cF_n})}(B)$$
is bounded by the logarithm of the number of the components of $\cA$. The proof is finished.
\end{proof}

\section{Approach of partial entropy}

In this section we give the proof of Theorem~\ref{main.D}.

For simplicity, we denote by $f_0=f$, $\mu_0=\mu$, $\cF_0=\cF$,
and the foliation boxes
$$\{(B^0_i,\Phi^0_i,D_i)\}_{i=1}^k=\{(B_i,\Phi_i,D_i)\}_{i=1}^k.$$
Let $\{(B^n_i,\Phi^n_i,D_i)\}_{i=1}^k$ be the foliation boxes of $\cF_n$ as in the
Definition~\ref{d.convergenceoffoliaiton}, and $\cA$ and $\cA^{\cF_n}$ be the partitions
constructed in the previous section.

\subsection{First approach:}
In the subsection, we use the partition $\cA^{\cF_n}$ to calculate the partial
metric entropy of $\mu_n$ along the expanding foliation $\cF_n$.

\begin{proposition}\label{p.approximate}
For every $n\geq 0$, we have
$$h_{\mu_n}(f_n,\cF_n)=\lim\frac{1}{m}H_{\mu_n}(\vee_{j=1}^m f_n^{-j}(\cA^{\cF_n})|\cA^{\cF_n})= \inf \frac{1}{m}H_{\mu_n}(\vee_{j=1}^m f_n^{-j}(\cA^{\cF_n})|\cA^{\cF_n}).$$

\end{proposition}

\begin{proof}
By the property of conditional entropy (\cite[Subsection 5.9]{Rok67}),
\begin{align*}
H_{\mu_n}(\vee_{j=1}^m f_n^{-j}(\cA^{\cF_n})|\cA^{\cF_n})=& H_{\mu_n}(f_n^{-1}(\cA^{\cF_n})\mid \cA^{\cF_n})+\cdots\\
& +H_{\mu_n}(f_n^{-m}(\cA^{\cF_n})\mid \vee_{j=0}^{m-1} f_n^{-j}(\cA^{\cF_n})).
\end{align*}

Because $\mu_n$ is $f_n$ invariant, it follows that
\begin{equation}\label{eq.approxiatedofpartition}
\begin{split}
H_{\mu_n}(\vee_{j=1}^m f_n^{-j}(\cA^{\cF_n})|\cA^{\cF_n}) & =H_{\mu_n}(\cA^{\cF_n}|f_n(\cA^{\cF_n}))+\dots\\
&\text{\;\;\;\;}+H_{\mu_n}(\cA^{\cF_n}|\vee_{j=0}^{m-1} f_n^{m-i}(\cA^{\cF_n})) \\
&=\sum_{j=1}^{m} H_{\mu_n}(\cA^{\cF_n}|\vee_{j=1}^i f_n^j(\cA^{\cF_n})).
\end{split}
\end{equation}

Since $\vee_{j=1}^i f_n^j(\cA^{\cF_n})$ is an increasing sequence, by
Lemmas~\ref{l.increasingpartition} and \ref{l.boundedconditionalentropy}, we have
$$H_{\mu_n}(\cA^{\cF_n}|\vee_{j=1}^i f_n^j(\cA^{\cF_n}))\searrow H_{\mu_n}(\cA^{\cF_n}|\vee_{j=1}^\infty f_n^j(\cA^{\cF_n}))= h_{\mu_n}(f_n,\cF_n).$$

Then by~\eqref{eq.approxiatedofpartition}:

$$\lim\frac{1}{m}H_{\mu_n}(\vee_{j=1}^m f_n^{-j}(\cA^{\cF_n})|\cA^{\cF_n})= \inf \frac{1}{m}H_{\mu_n}(\vee_{j=1}^m f_n^{-j}(\cA^{\cF_n})|\cA^{\cF_n})
=h_{\mu_n}(f_n,\cF_n).$$

\end{proof}

\subsection{Second approach}
In this subsection, we use the conditional entropy between two finite partitions to approach the partial
$\mu_n$-entropy of the expanding foliation $\cF_n$. We begin by the following easy observation, where
$\cA^{m}_n=\vee_{j=1}^m f^{-j}_n \cA$.

\begin{lemma}\label{l.finest}
For every $m>0$, $1\leq i \leq k$ and $x\in B_i$,
\begin{equation}\label{eq.localproductstructure}
\vee_{j=0}^m f_n^{-j}(\cA^{\cF_n})(x)= \cA^m_n(x)\cap \cA^{\cF_n}(x).
\end{equation}
\end{lemma}
\begin{proof}
Denote by $\cF_{n,\loc}(x)$ the local plaque of foliation $\cF_n(x)$ which contains $x$.

Suppose by contradiction that there is $y\in \cA^m_n(x)$ such that $y\in \cF_{n,\loc}(x)$
but $\vee_{j=0}^m f_n^{-j}(\cA^{\cF_n})(y)\neq \vee_{j=0}^m f_n^{-j}(\cA^{\cF_n})(x)$. Let $0<k\leq m$ be
the number such that
\begin{itemize}
\item $y_{j}=f^{j}(y)$ and $x_{j}=f^{j}(x)$ belong to the same elements of $\cA^{\cF_n}$ for every $0\leq j <k$;
\item $y_{k}$ and $x_{k}$ belong to different elements of $\cA^{\cF_n}$.
\end{itemize}

Because $\cA$ has small diameter, $\cA^{\cF_n}(y_{k-1})=\cA^{\cF_n}(x_{k-1})$ also has small diameter,
which implies that $f_n(\cA^{\cF_n}(y_{k-1}))$ is contained in $\cF_{n,\loc}(x_{k})$. Then by the
definition of $\cA^{\cF_n}$:
$$y_{k}\in \cF_{n,\loc}(x_k)\cap \cA(x_k)=\cA^{\cF_n}(x_k),$$
a contradiction to the assumption.
\end{proof}
\subsubsection{New partitions:}
Let $\cC_{i,1}\leq \cC_{i,2}\leq \dots$ be a sequence of finite partitions on $D^i$ such that
\begin{itemize}
\item[(A)] $\diam(\cC_{i,t})\to 0$;
\item[(B)] for any $i,t,m\geq 0$ and any element $C$ of $\cC_{i,t}$:
$\mu_n(\cup_{x\in \partial C}\cA^{\cF_n}(x))=0$.
\end{itemize}

For every $i,t\geq 0$, the partition $\cC_{i,t}$ induces a partition $\tilde{\cC}_{n, i,t}$ on the foliation box $B^n_i$:
$$\tilde{C}_{n,i,t}=\{\cup_{x\in C} \cA^{\cF_n}(x); \text{ $C$ is an element of $\cC_{i,t}$}\}.$$

\begin{remark}\label{r.radius}

Because $\diam(\cC_{i,t})\to 0$, for any $x\in B_n^i$, $\tilde{C}_{n,i,t}(x)\to \cA_{\cF_n}(x)$.

\end{remark}

For an element $P$ of $\cA^{m}_n$, suppose that $I\mid _P=i$ (see
\eqref{eq.definitionofI} on the definition of the function $I(x)$), which implies that $P\subset B_i^n$.
Then $\tilde{\cC}_{n,i,t}$ induces on $P$ a partition $P_t$: $\{P\cap \tilde{C};
\text{ $\tilde{C}$ is an element of $\cC_{n,i,t}$} \}$. And for every $m,n,t\geq 0$,
$$\cA^m_{n,t}=\{P_t; \text{ where $P\in \cA^{m}_n$}\}$$
is a new partition of the ambient manifold $M$. In the following we identify some
properties for the new partition, which are important for the further proof.

\begin{lemma}\label{l.newpartition}
For any $n,m\geq 0$:
\begin{itemize}
\item[(i)] $\cA^m_{n,t}\underset{t\to \infty}{\nearrow} \vee_{j=0}^m f_n^{-i}\cA^{\cF_n}$;
\item[(ii)] $\cA^m_n<\cA^{m}_{n,t}<\vee_{j=0}^m f_n^{-i}\cA^{\cF_n}$;
\item[(iii)] $\mu_n(\partial \cA^m_{n,t})=0$.
\end{itemize}
\end{lemma}
\begin{proof}
From the construction of the partition $\cA^m_{n,t}$, (i) and (ii) follow immediately. Moreover,
$$\partial \cA^m_{n,t}\subset \partial \cA^m_n\bigcup \cup_{i=1}^k \cup_{C\in \cC_{i,t}} \cup_{x\in \partial C}\cA^{\cF_n}(x).$$

By the assumption (B) above, $\mu_n(\cup_{x\in \partial C}\cA^{\cF_n}(x))=0$. Note also that by \eqref{eq.smallboundary}, $\mu_n(\partial \cA^m_n)=0$. The proof is finished.

\end{proof}

The following proposition is the key for the approach:

\begin{proposition}\label{p.limit}

$H_{\mu_n}(\cA^m_{n,t}\mid \cA^0_{n,t})\underset{t\to \infty}{\searrow} H_{\mu_n}(\vee_{j=0}^m f_n^{-j}\cA^{\cF_n}\mid \cA^{\cF_n})$.

\end{proposition}
\begin{proof}
Applying Lemma~\ref{l.newpartition} (i) on $m=0$, $\cA^0_{n,t} \nearrow \cA^{\cF_n}$. Because both partitions
$\cA^{m}_{n,t}$ and $\cA^0_{n,t}$ are finite, $H_{\mu_n}(\cA^m_{n,t}\mid \cA^0_{n,t})<\infty$. By Lemma~\ref{l.increasingpartition}, we have
$$H_{\mu_n}(\cA^m_{n,t}\mid \cA^0_{n,t}){\searrow} H_{\mu_n}(\cA^m_{n,t}\mid \cA^{\cF_n}) =H_{\mu_n}(\cA^m_{n,t}\vee \cA^{\cF_n}\mid \cA^{\cF_n}).$$

Then the proof follows from the claim that $\cA^m_{n,t}\vee \cA^{\cF_n}=\vee_{j=0}^m f_n^{-j}\cA^{\cF_n}$.

It remains to prove the claim, which is a corollary of Lemma~\ref{l.newpartition} (ii):
On one hand, $\cA^{m}_{n,t}<\vee_{j=0}^m f_n^{-i}\cA^{\cF_n}$, which implies that
\begin{equation}\label{eq.equalityupperboundary}
\cA^m_{n,t}\vee \cA^{\cF_n}<\vee_{j=0}^m f_n^{-i}\cA^{\cF_n}\vee \cA^{\cF_n}=\vee_{j=0}^m f_n^{-j}\cA^{\cF_n}.
\end{equation}
On the other hand, $\cA^m_n<\cA^{m}_{n,t}$. Then
$$\cA^m_n\vee \cA^{\cF_n}<\cA^m_{n,t}\vee \cA^{\cF_n}.$$
By Lemma~\ref{l.finest},
\begin{equation}\label{eq.equalitylowerboundary}
\vee_{j=0}^m f_n^{-j}\cA^{\cF_n}<\cA^m_{n,t}\vee \cA^{\cF_n}.
\end{equation}
We conclude the proof of the claim by \eqref{eq.equalityupperboundary} and \eqref{eq.equalitylowerboundary}.
\end{proof}

\begin{corollary}\label{c.entropyoff_n}
$H_{\mu_n}(\cA^m_{n,t}\mid \cA^0_{0,t})> h_{\mu_n}(f_n,\cF_n)$.
\end{corollary}
\begin{proof}
This is a consequence of Propositions~\ref{p.approximate} and \ref{p.limit}.
\end{proof}
\subsection{Proof of Theorem~\ref{main.D}}

By Proposition~\ref{p.approximate}, for any $\vep>0$, there is $m_0$ sufficiently large, such that
\begin{equation}\label{e.semiupper 1}
\frac{1}{m_0}H_{\mu}(\vee_{j=0}^{m_0}f^{-j}\cA^{\cF}|\cA^{\cF})-\frac{\vep}{3}\leq h_\mu(f,\vep)
\end{equation}
By Proposition~\ref{p.limit}, we may further take $t_0>0$ large, such that
\begin{equation}\label{e.semiupper 2}
H_{\mu}(\cA^{m_0}_{0,t_0}\mid \cA^0_{0,t_0})-\frac{\vep}{3}\leq H_{\mu}(\vee_{j=0}^{m_0} f^{-j}\cA^{\cF}\mid \cA^{\cF})
\end{equation}

Because $f_n$ converge to $f$ in the $\C^1$ topology, and the foliations $\cF_n$ converge
to $\cF$ (see Definition~\ref{d.convergenceoffoliaiton}), each component of
$P_0\in \cA^{m_0}_{0,t_0}$ is converged by the corresponding component $P_n$ of the partition $\cA^{m_0}_{n,t_0}$
in the Hausdorff topology. Note that $\mu_n$ converge to $\mu$ in the weak* topology, and by
Lemma~\ref{l.newpartition} (iii), $\mu(\partial P_0)=0$, hence,
$$\lim_{n\to \infty}\mu_n(P_n)= \mu(P_0).$$

Because $\cA^{m_0}_{0,t_0}$ is a finite partition, we have
\begin{equation}
\lim_{n\to \infty}H_{\mu_n}(\cA^{m_0}_{n,t_0}\mid \cA^0_{n,t})= H_{\mu}(\cA^{m_0}_{0,t_0}\mid \cA^0_{0,t_0}).
\end{equation}

Then there is $n_0$ large enough, such that for any $n\geq n_0$,
$$H_{\mu_n}(\cA^{m_0}_{n,t_0}\mid \cA^0_{n,t_0})-\frac{\vep}{3}\leq H_{\mu}(\cA^{m_0}_{0,t_0}\mid \cA^0_{0,t_0}).$$
Combining~\eqref{e.semiupper 1} and~\eqref{e.semiupper 2}, for any $n\geq n_0$, one has
$$H_{\mu_n}(\cA^{m_0}_{n,t_0}\mid \cA^0_{n_0,t_0})-\vep \leq h_{\mu}(f,\cF).$$

By Corollary~\ref{c.entropyoff_n}, for any $n\geq n_0$,
$$h_{\mu_n}(f_n,\cF_n)-\vep \leq h_{\mu}(f,\cF).$$
Because $\vep$ can be taken arbitrarily small, we conclude the proof of Theorem~\ref{main.D}.

\section{Gibbs $u$-states of partially hyperbolic diffeomorphisms}
Let $f$ be a $\C^{1+}$ partially hyperbolic diffeomorphism with invariant splitting on the tangent bundle:
$T_xM=E^s_x\oplus E^c_x\oplus E^u_x$. Denote by $\cF^u$ the unstable foliation of $f$ which is tangent
to the unstable bundle $E^u$, and write $\Jac^u(x)=\det Df\mid_{E^u_x}$.

\subsection{Preliminaries for Gibbs $u$-states}

Denote by $\Gibb^u(f)$ the set of Gibbs $u$-states of $f$. The proofs for the following
basic properties of Gibbs $u$-states can be found in Bonatti, D\'{i}az and Viana~\cite[Subsection 11.2]{BDVnonuni}:
\begin{proposition}\label{p.Gibbsustates}
\begin{itemize}
\item[(1)] $\Gibb^u(f)$ is non-empty, weak* compact and convex. Ergodic components of Gibbs $u$-states are Gibbs u-states.
\item[(2)] The support of every Gibbs $u$-state is $\cF^u$-saturated, that is, it consists of entire strong
    unstable leaves.
\item[(3)] For Lebesgue almost every point $x$ in any disk inside some strong unstable leaf, every accumulation point of
    $\frac{1}{n}\sum_{j=0}^{n-1}\delta_{f^j(x)}$ is a Gibbs $u$-state.
\item[(4)] Every physical measure of $f$ is a Gibbs $u$-state and, conversely, every ergodic $u$-state whose center Lyapunov
    exponents are negative is a physical measure.
\end{itemize}
\end{proposition}

The following upper bound for the partial entropy along the unstable foliation $\cF^u$ follows \cite{LY85a, LY85b}.

\begin{proposition}\label{p.Ruelle}
Let $\mu$ be an invariant probability measure of $f$, then
$$h_\mu(f,\cF^u)\leq\int \log \Jac^u(x) d\mu(x).$$
Moreover,
\begin{equation}\label{e.Pesinformula}
h_\mu(f,\cF^u)=\int \log \Jac^u(x) d\mu(x).
\end{equation}
if and only if $\mu$ is a Gibbs $u$-state of $f$.
\end{proposition}
\begin{proof}
The inequality follows by~\cite[Theorem $C^\prime$]{LY85b}, when $f$ is $\C^2$.
It was pointed out by \cite{Br} that the same inequality goes well for $\C^{1+}$
diffeomorphism.

The second part was stated in \cite[Theorem 3.4]{L84}.

\end{proof}
The following equality was built in~\cite[Proposition 5.1]{LY85b}, when $\mu$ is $\C^2$.
As explained above, which also holds for general situation:

\begin{proposition}\label{p.vanishingcenterexponents}

Let $\mu$ be a probability measure of $f$ such that the center exponents of $\mu$ are all
non-positive, then
$$h_\mu(f,\cF^u)=h_\mu(f).$$

\end{proposition}

\subsection{Diffeomorphisms with mostly expanding/contracting center}

In this Subsection we state equivalent definition for diffeomorphisms with mostly contracting (resp. expanding)
center direction which was proved in \cite{And10} (resp. \cite{AC}). For completeness, we provide their proofs.

\begin{proposition}\label{p.mc}[\cite{And10}]
A diffeomorphism $f$ has mostly contracting center if and only if there is
$N\in \mathbb{N}$ and $b>0$ such that for any $\mu\in \Gibb^u(f)$,

\begin{equation}\label{eq.definitionofmc}
\int \log \|Df^N|_{E^{cs}(x)}\|d\mu(x)<-b.
\end{equation}
\end{proposition}

\begin{proposition}\label{p.me}[\cite{AC}]
A diffeomorphism $f$ has mostly expanding center if and only if there is
$N\in \mathbb{N}$ and $b>0$ such that for any $\mu\in \Gibb^u(f)$,
\begin{equation}\label{eq.weakdefinitionofme}
\int \log \|Df^{-N}|_{E^{cs}(x)}\|d\mu(x)<-b.
\end{equation}
\end{proposition}

\begin{proof}[Proof of Proposition~\ref{p.mc}:]
Suppose $f$ has mostly contracting center. Then for every Gibbs $u$-state $\mu$
of $f$, the integration of the largest center exponent of $\mu$ is negative, that is,
$$\lim_n \int \frac{1}{n}\log \|Df^n\mid_{E^{cs}(x)}\|d\mu<0.$$
There exists $N_\mu>1$ and $b_\mu>0$ such that
\begin{equation}\label{eq.singlemeasure}
\int \log \|Df^{N_\mu}\mid_{E^{cs}}\|d\mu < -N_\mu b_\mu.
\end{equation}
We may take a neighborhood $\cV_\mu$ of $\mu$ inside the probability measure space of $M$, such that,
\eqref{eq.singlemeasure} holds for any probability measure $\nu\in \cV_\mu$:
\begin{equation}\label{eq.neighborhoodmeasures}
\int \log \| Df^{N_\mu}\mid_{E^{cs}}\|d\nu \leq -N_\mu b_\mu.
\end{equation}

Because the space of Gibbs $u$-states of $f$ is compact (see
Proposition~\ref{p.Gibbsustates} (1)), there is a finite open covering
$\{\cV_{\mu_j}\}_{j=1}^k$ of $\Gibb^u(f)$. For simplicity, we write $N_j=N_{\mu_j}$ and
$b_j=b_{\mu_j}$. Let $N=\prod_{j=1}^k N_j$ and $b=\min\{b_1,\dots, b_k\}$. For any Gibbs $u$-state
$\mu$ of $f$, there is $1\leq j_0\leq k$ such that $\mu \in \cV_{j_0}$. Because
$$Df^N\mid_{E^{cs}(x)}=Df^{N_{j_0}}\mid_{E^{cs}(f^{N-N_{j_0}}(x))} \circ\dots \circ Df^{N_{j_0}}\mid_{E^{cs}(x)},$$
 by \eqref{eq.neighborhoodmeasures},
\begin{equation}\label{eq.allmeasures}
\begin{split}
\int \log \|Df^{N}\mid_{E^{cs}}\|d\mu &\leq \frac{N}{N_{j_0}}\int \log\|Df^{N_{j_0}}\mid_{E^{cs}}\|d\mu\\
&\leq -\frac{N}{N_{j_0}} b_{j_0} \leq -b.
\end{split}
\end{equation}

On the contrary, now we assume that \eqref{eq.definitionofmc} holds for any Gibbs $u$-state $\mu$
of $f$, which implies that the center exponents of any ergodic Gibbs $u$-state of $f$ are all negative.
Because the ergodic components of every Gibbs $u$-state are still Gibbs $u$-states (see
Proposition~\ref{p.Gibbsustates} (1)), we then conclude that for any Gibbs $u$-state $\mu$ of $f$,
the center exponents of $\mu$ almost every point are all negative. The proof of Proposition~\ref{p.mc} is complete.
\end{proof}

The proof of Proposition~\ref{p.me} is quite similar: we only need replace the diffeomorphism $f$ in the
proof above by its inverse $f^{-1}$. We will not detail the proof here.

\subsection{Proof of Theorem~\ref{main.A}}
Instead of proving Theorem~\ref{main.A}, we will prove the following equivalent proposition:

\begin{proposition}\label{p.mainA'}

Let $\{f_n\}_{1}^\infty$ be a sequence of $\C^{1+}$ partially hyperbolic diffeomorphisms,
and $\mu_n$ Gibbs-$u$ state of $f_n$. Suppose $f_n$ converge to a $\C^{1+}$ diffeomorphism
$f$ in the $\C^1$ topology, and $\mu_n$ converge to a probability measure $\mu$ in the weak-*
topology, then $\mu$ is a Gibbs-$u$ state of $f$.

\end{proposition}

Denote by $\Jac_n(x)= \det Df_n\mid_{E^u_x}$, and $\cF^u_n$ the unstable foliation of $f_n$.
Because $\mu_n$ is a Gibbs $u$-state of $f_n$ for each $n$, by Proposition~\ref{p.Ruelle},
$$h_{\mu_n}(f_n,\cF^u_n)=\int \log\Jac_{n}(x)d\mu_n(x).$$

By the unstable manifold theorem of \cite{HPS}, the foliations $\cF^u_n$ converge to $\cF^u$
as in Definition~\ref{d.convergenceoffoliaiton}. Hence, as a corollary of Theorem~\ref{main.D},
$$\limsup_{n\to \infty} h_{\mu_n}(f_n,\cF^u_n)\leq h_\mu(f,\cF^u).$$

Note that $\{\Jac_n(\cdot )\}$ are continuous functions which are bounded from below by zero and
converge uniformly to $\Jac(\cdot)$, we have
$$\lim_{n\to \infty} \int \log\Jac_n(x)d\mu_n(x)=\int \log \Jac(x)d\mu(x).$$

Therefor, $h_{\mu}(f,\cF^u)\geq \int \log\Jac(x)d\mu(x)$.
But by the first part of Proposition~\ref{p.Ruelle},
$$h_{\mu}(f,\cF^u)\leq \int \log\Jac(x)d\mu(x).$$

Hence, we have the equality
$$h_{\mu}(f,\cF^u)= \int \log\Jac(x)d\mu(x).$$
By the second part of Proposition~\ref{p.Ruelle}, $\mu$ is a Gibbs $u$-state of $f$. The proof is complete.

\subsection{Proof of Theorem~\ref{main.B}}
We begin by showing that the set of diffeomorphisms with mostly contracting center is $\C^1$
open:

Let $f$ be a $\C^{1+}$ partially hyperbolic diffeomorphism with mostly contracting center. By
Proposition~\ref{p.mc}, there is $N>1$ and $b>0$ such that for any Gibbs $u$-state $\mu$ of $f$,
\begin{equation}\label{eq.theoremBmc}
\int \log \|Df^N\mid_{E^{cs}} \|d\mu <-b.
\end{equation}

By Theorem~\ref{main.A}, for any small neighborhood $\cV$ of $\Gibb^u(f)$ in the space of probability
measures of $M$, there is a $\C^1$ neighborhood $\cU$ of $f$, such that for any $\C^{1+}$ diffeomorphism $g\in \cU$,
$\Gibb^u(g)\subset \cV$. Because the bundles $E^{cs}$ varies continuously with respect to the diffeomorphisms
in the $\C^1$ topology, we may take the neighborhood $\cV$ of $\Gibb^u(f)$ sufficiently small, and $\cU$ the
neighborhood of $f$ small enough, such that for any $\C^{1+}$ diffeomorphism $g\in \cU$,
$$\int \log \|Dg^N\mid_{E^{cs}} \|d\mu <-b \text{ for any Gibbs $u$-state $\mu$ of $g$}.$$
Then by Proposition~\ref{p.mc}, every $\C^{1+}$ diffeomorphism $g\in \cU$ has mostly contracting center direction.

Replacing the diffeomorphism $f$ in the above proof by $f^{-1}$, and making use of Proposition~\ref{p.me} instead of
Proposition~\ref{p.mc}, one may show that the set of partially hyperbolic diffeomorphisms with mostly expanding center
is $\C^1$ open.

\section{Diffeomorphisms with mostly expanding center}
We prove Theorem~\ref{main.C} in this section. The proof is divided into two parts: In
Subsection~\ref{ss.uniformbasin} we show the existence of physical measures for diffeomorphisms with
mostly expanding center, moreover, these physical measures admit uniform size of basins in a robust way.
In Subsection \ref{ss.CI} we prove Theorem~\ref{main.C}.

Throughout this section, let $f: M\to M$ be a $\C^{1+}$ partially hyperbolic diffeomorphism.

\subsection{New description of diffeomorphisms with mostly expanding center}
In this section we give an improved version of Proposition~\ref{p.me} for diffeomorphisms with mostly expanding center:

\begin{proposition}\label{p.onestep}
Suppose $f$ has mostly expanding center, then there is $N_0\in \mathbb{N}$ and $b>0$ such that,
for any $\tilde{\mu}\in \Gibb^u(f^{N_0})$,
\begin{equation}\label{eq.definitionofme}
\int \log \|Df^{-N_0}|_{E^{cs}(x)}\|d\tilde{\mu}(x)<-b.
\end{equation}
\end{proposition}

\begin{remark}\label{rk.robustonestep}
By a proof similar to Theorem~\ref{main.A}, one can show that \eqref{eq.definitionofme} holds for any Gibbs $u$ states
of any $C^{1+}$ diffeomorphism $g$ which is $C^1$ close to $f$.
\end{remark}

The difference between the above proposition with Proposition~\ref{p.me} is that, a Gibbs $u$-state of $f^N$ for some
$N>0$ is not necessary a Gibbs $u$-state of $f$.

\begin{proof}
By Proposition~\ref{p.me}, there is $N>0$ and $b>0$ such that for any Gibbs $u$-state $\mu$ of $f$,
$$\int \log \|Df^{-N}|_{E^{cs}(x)}\|d\mu(x)<-b.$$

Let $C=\max \log \|Df\|+\max \log \|Df^{-1}\|$. Take any $k> \frac{2CN}{b}+1$ and let $N_0=Nk$.

By (1) of Proposition~\ref{p.Gibbsustates}, we only need to consider $\tilde{\mu}$ that is ergodic with respect to $f^{N_0}$. Also by Birkhoff ergodic
theorem, it suffices for us to show that for $\tilde{\mu}$ almost every point $x$,
\begin{equation}\label{eq.pointwiseexpanding}
\limsup_{n\to \infty} \frac{1}{n}\sum_{j=1}^n \log \|Df^{-N_0k}\mid_{E^{cu}(f^{jN_0}(x))}\|<-b.
\end{equation}

Let $\mu=\frac{1}{N_0}\sum_{i=0}^{N_0-1}(f^i)_*(\tilde{\mu})$; it is an ergodic measure of $f$. From the definition of Gibbs $u$-state,
the conditional probabilities of $\tilde{\mu}$ along strong unstable leaves are absolutely continuous
with respect to the Lebesgue measure on the unstable leaves, so do the conditional probabilities of $\mu$ along strong unstable leaves,
this implies that $\mu$ is an ergodic Gibbs $u$-state of $f$. In the following we will prove \eqref{eq.pointwiseexpanding} holds for
$\mu$ almost every $x$.

Consider $\nu_0$ an ergodic decomposition of $\mu$ respect to the map
$f^{N}$, denote by $\nu_j=(f^j)_* \nu$, then $\nu_j=\nu_{j+N}$ and $\mu=\frac{1}{N}\sum_{j=0}^{N-1}\nu_j$.

By \eqref{eq.uniformme}, there is some $0\leq j\leq N-1$ such that
$$\int \log \|Df^{-N}|_{E^{cs}(x)}\|d\nu_j(x)<-b.$$
For simplicity,
we assume that $j=0$. Then for every point $x$ in the basin of $\mu$,
there is $0\leq j_0\leq N-1$, such that
$$\lim_{n}\frac{1}{n}\sum_{t=0}^{n-1} (f^{tN})_*\delta_{x_{j_0}}=\nu_0,$$
where $x_{j_0}=f^{j_0}(x)$. This implies that,
$$\limsup_{n\to \infty} \frac{1}{n}\sum_{j=1}^n \log \|Df^{-N}\mid_{E^{cu}(f^{jN}(x_{j_0}))}\|<-b.$$
Replacing $N$ by $N_0=kN$, we have
\begin{equation}\label{eq.largeiteration}
\limsup_{n\to \infty} \frac{1}{n}\sum_{j=1}^n \log \|Df^{-N_0}\mid_{E^{cu}(f^{jN_0}(x_{j_0}))}\|<-kb.
\end{equation}
Hence,
\begin{equation}
\begin{split}
\log \|Df^{-N_0}\mid_{E^{cu}(x_{j_0})}\| &= \log \|Df^{-j_0}\circ Df^{-N_0}\circ Df^{j_0}\mid_{E^{cu}(x)}\| \\
&\geq \log\|Df^{-N_0}\mid_{E^{cu}(x)}\|-2Cj_0\\
&\geq \log\|Df^{-N_0}\mid_{E^{cu}(x)}\|-2CN.
\end{split}
\end{equation}
Combining this inequality with \eqref{eq.largeiteration} and the choice of $k$, we obtain
$$\limsup_{n\to \infty} \frac{1}{n}\sum_{j=1}^n \log \|Df^{-N_0}\mid_{E^{cu}(f^{jN_0}(x))}\|<-kb+2CN <-b.$$

\end{proof}

\subsection{Gibbs $cu$-states}
In this subsection we recall the main result of \cite{ABV00} on the existence of physical measures for the
diffeomorphisms which are no-uniformly expanding along the center-unstable direction (see the definition below).
We will outline the arguments and explain that these physical measures indeed have uniform size of basins
for nearby diffeomorphisms.

\begin{definition}\label{d.nonuniformlyexpanding}
For $b>0$, we say $f$ is \emph{$b$ no-uniformly expanding along the center-unstable direction},
if
\begin{equation}
\limsup_{n\to \infty} \frac{1}{n}\sum_{j=1}^n \log \|Df^{-1}\mid_{E^{cu}(f^j(x))}\|<-b<0,
\end{equation}
for Lebesgue almost every $x\in M$.
\end{definition}

\begin{theorem}[\cite{ABV00} Theorem A]
Assume that $f$ is \emph{no-uniformly expanding along the center-unstable direction}.
Then $f$ has finitely many physical measures, whose basins cover a full volume subset of the
ambient manifold.
\end{theorem}

In the blow, we give more precise description of these physical measures.

\subsubsection{Hyperbolic time}
\begin{definition}
Given $b>0$, we say $n$ is an $b$-\emph{hyperbolic time} for a point $x$ if
$$\frac{1}{k}\sum_{j=n-k+1}^n\log\|Df^{-1}\mid_{E^{cu}(f^j(x))}\|\leq -b \text{ for any $0<j \leq n$}.$$
\end{definition}

Let $D$ be any $\C^1$ disk, we use $d_D(\cdot,\cdot)$ denotes the distance between two points in
the disk.

\begin{lemma}[\cite{ABV00} Lemma 2.7]\label{l.hyperbolictime}
For any $b>0$, there is $\delta_1>0$ such that, given any $\C^1$ disk $D$ tangent to the
center-unstable cone field, $x\in D$ and $n\geq 1$ an $b/2$-hyperbolic time for $x$, then
$$d_{f^{n-k}(D)}(f^{n-k}(y,f^{n-k}(x)))\leq e^{-kb/2} d_{f^n(D)}(f^n(x),f^n(y)),$$
for any point $y\in D$ with $d_{f^n(D)}(f^n(x),f^n(y))\leq \delta_1$.
\end{lemma}

\begin{remark}\label{r.uniformunstable}
For fixed $b>0$, we can take $\delta_1$ to be constant for the diffeomorphisms
in a $\C^1$ neighborhood of $f$.
\end{remark}

\subsubsection{Physical measures with uniform size of basins}
Suppose $f$ is $a$ non-uniformly expanding along the center-unstable direction.
By \cite[Lemma 4.5]{ABV00}, $f$ admits an ergodic probability measure $\mu$
such that the conditional measures of $\mu$ along a family of local unstable manifolds are absolutely
continuous with respect to Lebesgue measure. Moreover, by \cite[Proposition 4.1]{ABV00},
the size of this unstable manifolds are larger than $\delta_1/4$, where $\delta_1$ is obtained
by Lemma~\ref{l.hyperbolictime}.

Then the basin of $\mu$ contains a full Lebesgue measure subset of a local unstable manifold.
Because the basin is saturated by stable leaves, and the stable foliation is absolute continuous,
$\mu$ is indeed a physical measure, whose basin contains Lebesgue almost every point of a ball
with radius $\delta_1/4$.

And by \cite[Corollary 4.6]{ABV00}, $f$ admits finitely many such physical measures, and the union
of their basins has full volume. Combine this with Remark~\ref{r.uniformunstable}, we have that:

\begin{proposition}\label{p.custate}
Suppose $f\in \Diff^{1+}(M)$ is $a$ non-uniformly expanding along the center-unstable direction, and
$\delta_1$ is obtained by Lemma~\ref{l.hyperbolictime}. Then there is a $\C^1$ neighborhood $\cU$
of $f$, such that for any $\C^{1+}$ diffeomorphism $g\in \cU$, if $g$ is also $a$ non-uniformly expanding
along the center-unstable direction, then it admits finite many physical measures. The basin of each
physical measure contains Lebesgue almost every point of a ball with radius $\delta_1/4$, and the union
of these basins has full volume.

\end{proposition}

\subsection{Basins with uniform size \label{ss.uniformbasin}}
The main result of this subsection is the following proposition.

\begin{proposition}\label{p.culargebasin}
Suppose $f$ has mostly expanding center, then there is a $\C^1$ neighborhood $\cU$ of $f$ and $\delta>0$,
such that every $\C^{1+}$ diffeomorphism $g\in \cU$ admits finitely many physical measures,
and the union of the basins has full volume. Moreover, the basin of each physical measure contains Lebesgue
almost every point of some ball with radius $\delta$.

\end{proposition}
\begin{proof}
By Proposition~\ref{p.onestep} and Remark~\ref{rk.robustonestep}, there exists a $\C^1$ neighborhood $\cU$ of $f$,
$b>0$ and $N_0\in \mathbb{N}$, such that for any $\C^{1+}$ diffeomorphism $g\in \cU$ and any
$\mu\in \Gibb^u(g^{N_0})$,
\begin{equation}\label{eq.uniformme}
\int \log \|Dg^{-N_0}|_{E^{cs}(x)}\|d\mu(x)<-b.
\end{equation}

We will show that:
\begin{lemma}\label{l.largeiteration}
$g^{N_0}$ is $a$ no-uniformly expanding along the center-unstable direction, that is,
for Lebesgue almost every $x\in M$, we have
\begin{equation}\label{eq.eventiallyexpanding}
\limsup_{n\to \infty} \frac{1}{n}\sum_{j=1}^n \log \|Dg^{-N_0}\mid_{E^{cu}(g^{jN_0}(x))}\|<-b<0.
\end{equation}
\end{lemma}

\begin{proof}
By Proposition~\ref{p.Gibbsustates}, for Lebesgue almost every point $x$ in any disk inside some strong unstable
leaf, every accumulation point of $\frac{1}{n}\sum_{j=0}^{n-1}\delta_{g^{jN_0}(x)}$ is a Gibbs u-state of $g^{N_0}$. Suppose
$\mu$ is the limit point of a subsequence $\frac{1}{n_0}\sum_{j=0}^{n_0-1}\delta_{g^{jN_0}(x)}$, then
$$\lim_{i\to \infty} \frac{1}{n_i}\sum_{j=1}^{n_i} \log \|Dg^{-N_0}\mid_{E^{cu}(g^{jN_0}(x))}\|=\int \log \|Dg^{-N_0}|_{E^{cs}(x)}\|d\mu(x)<-b<0.$$

Because the subsequence can be chosen arbitrarily, we conclude that
$$\limsup_{n\to \infty} \frac{1}{n}\sum_{j=1}^n \log \|Dg^{-N_0}\mid_{E^{cu}(g^{jN_0}(x))}\|<-b<0.$$

Since $g$ is $C^{1+}$, its unstable foliation is absolutely continuous, the points satisfy \eqref{eq.eventiallyexpanding}
have full volume.
\end{proof}

Let us continue the proof of Proposition~\ref{p.culargebasin}.

By Lemma~\ref{l.largeiteration} and Proposition~\ref{p.custate}, there is $\delta>0$
such that for every $\C^{1+}$ diffeomorphism $g\in \cU$, $g^{N_0}$ admits finite many physical measures
$\nu_{g,1},\dots, \nu_{g,i(g)}$, such that the basin of each physical measure contains Lebesgue almost every point
of a ball with radius $\delta$, and the union of the basins covers a full volume subset.
We conclude the proof of this proposition by the following observation: for each $j=1,\dots, i(g)$,
$$\mu=\frac{1}{N_0}\sum_{k=0}^{N_0-1} (g^k)_* \nu_{g,j}$$
is a physical measure of $g$, whose basin contains the basin of $\nu_{g,j}$ for the map $g^{N_0}$.

\end{proof}

\subsection{Proof of Theorem~\ref{main.C}:\label{ss.CI}}

Now we consider $f$ to be a $\C^{1+}$ volume preserving partially hyperbolic diffeomorphism with
one-dimensional center. By \cite{BW}, the Lebesgue measure is ergodic, and by Birkhoff ergodic
theorem, it is a physical measure and the basin has full volume. We further assume that
\begin{itemize}
\item[(a)] the center exponent of the Lebesgue measure is positive;

\item[(b)] $f$ is accessible.
\end{itemize}

We are going to show that there is a $\C^1$
neighborhood $\cU$ of $f$, such that every $\C^{1+}$ diffeomorphism $g\in \cU$ has mostly expanding center,
with a unique physical measure, and the basin has full volume. The proof consists of a series of lemmas.

\begin{lemma}\label{l.singleme}

$f$ has mostly expanding center.

\end{lemma}
\begin{proof}
We prove by contradiction. Suppose that $f$ admits a Gibbs $u$-state $\mu$ with non-positive center exponent,
that is,
$$\int \lambda^c(x)d\mu(x)=\int \log\|df\mid_{E^c(x)}\|d\mu(x)\leq 0.$$
By (2) of Proposition~\ref{p.Gibbsustates}, we may assume that $\mu$ is ergodic.

Since the Lebesgue measure, denoted by $\Leb$, has positive center exponent, $\mu \neq \Leb$. We claim that the center exponent of $\mu$ cannot be negative,
otherwise by (4) of Proposition~\ref{p.Gibbsustates},
$\mu$ is a physical measure; this contradicts that the basin of the Lebesgue measure has full volume.

It remains to show that the center exponent of $\mu$ cannot be vanishing. Assume that this is not the case, and denote by $\Lambda^u$ and $\Lambda^s$ the sum of the Lyapunov exponents of measure $\mu$
on the bundles $E^u$ and $E^s$ respectively, then
$$\Lambda^u=\int \log \det Df\mid_{E^u(x)} d\mu (x) \text{ and } \Lambda^s=\int \log \det Df\mid_{E^s(x)} d\mu (x).$$
Because $f$ is volume preserving,
\begin{equation}\label{eq.volumepreserving}
\Lambda^u+\Lambda^s=0.
\end{equation}

Then by Propositions~\ref{p.vanishingcenterexponents} and~\ref{p.Ruelle}, $\mu$ satisfies the entropy formula, that is,
$$h_{\mu}(f)=h_{\mu}(f,\cF^u)=\Lambda^u.$$

Because the metric entropy for $f$ and $f^{-1}$ coincide, $h_\mu(f)=h_{\mu}(f^{-1})$.
By Proposition \ref{p.vanishingcenterexponents}, $\mu$ also satisfies the entropy
formula for $f^{-1}$:
$$h_{\mu}(f^{-1},\cF^s)=h_{\mu}(f^{-1})=h_\mu(f)=-\Lambda^s,$$
Applying Proposition \ref{p.Ruelle} on $f^{-1}$, the above equality implies that $\mu$ is a
Gibbs $u$-state of $f^{-1}$.  By (1) of Proposition~\ref{p.Gibbsustates}, the support
of $\mu$ is $\cF^u$ and $\cF^s$ saturated. Since $f$ is accessible, $\supp(\mu)$ coincides
with the whole manifold $M$.

Recall that both $\Leb$ and $\mu$ are Gibbs $u$-states for $f^{-1}$, the conditional measure along the
unstable leaves of $f^{-1}$ inside each foliation box are equivalent to the Lebesgue measure
on the corresponding leaves.

We first take a $\Leb$-typical unstable plaque $D$ for $f^{-1}$, which means that, Lebesgue almost
every point of $D$ is the regular point of $\Leb$ for $f^{-1}$. In particular, they admit Pesin stable manifolds with
dimension $i_s+1$, where $i_s=\dim(E^s)$, and these Pesin stable manifolds are absolutely continuous and their union, denoted
by $\Gamma^s$ also belongs to the basin of $\mu$.

Because $\supp(\mu)=M$, we may take a $\mu$-typical unstable disk $D^\prime$ which is sufficiently close to
$D$ such that Lebesgue almost every point of $D^\prime$ belongs to the basin of $\mu$. But
$\Gamma^s$ intersects $D^\prime$ with a positive Lebesgue measure subset and the intersection belongs to
the basin of $\Leb$, a contradiction.
\end{proof}

\begin{lemma}
There is a $\C^1$ neighborhood of $f$ such that every $\C^{1+}$ diffeomorphism in this neighborhood
admits a unique physical measure, whose basin has full volume.

\end{lemma}
\begin{proof}

By Lemma~\ref{l.singleme} and Proposition~\ref{p.culargebasin}, there is  a $\C^1$ neighborhood $\cU$ of $f$ and $\delta>0$ such that
every $\C^{1+}$ diffeomorphism $g\in \cU$ satisfies the following properties:
\begin{itemize}
\item $g$ has mostly expanding center;
\item $g$ has finitely many physical measures, the union of the basins of physical measures have full volume;
\item each basin contains Lebesgue almost every point of a ball with radius $\delta$.
\end{itemize}

Because $f$ is transitive (by the ergodicity), we may take an open neighborhood $\cU_0\subset \cU$ of $f$,
such that every diffeomorphism contained in $\cU_0$ is \emph{$\delta$ transitive}, that is, the positive
iteration of every ball with radius $\delta$ intersects any other ball with radius $\delta$.

Now we claim that every $\C^{1+}$ diffeomorphism $g\in \cU_0$ admits a unique physical measure. If this is not the case, then the iteration of the basin of one physical measure will intersect the basin of a distinct physical measure.
Note that the basin is invariant under iteration, then the basins of the two physical measures have non-trivial
intersection, a contradiction. The proof of Theorem~\ref{main.C} is finished.
\end{proof}

\section{Physical measures for $C^1$ generic diffeomorphisms\label{s.generic}}
Throughout this section, we assume $f$ satisfies the hypothesis of Corollary~\ref{maincor.generic}, that is,
$f$ is a $C^{1+}$ accessible volume preserving partially hyperbolic diffeomorphism
with one-dimensional center and positive center exponent. Let $\cU$ be the small $C^1$ neighborhood of $f$ provided in
Theorem~\ref{main.C}.

Let us briefly sketch  the proof of Corollary~\ref{maincor.generic}. We will introduce a class of candidates
of physical measures for any $C^1$ partially hyperbolic diffeomorphism $f$, which we denote by $G(f)$ (see Proposition~\ref{p.physical}).
Then we consider the map $\cG(\cdot)$ which maps a diffeomorphism $f\in \cU$ to $G(f)$. We
show that the map is upper semi-continuous (Proposition~\ref{p.Gsemicontinuous}) and restricted on
$C^{1+}$ diffeomorphisms, its image consists of a unique measure (Proposition~\ref{p.Gunique}). Hence by a classical generic argument,
we conclude the proof.

Fix any $C^1$ diffeomorphism $g\in \cU$. In the proof we will consider two space of invariant
probabilities of $g$:
\begin{definition}\label{df.G}
\begin{itemize}
\item[(A1)] \begin{equation*}\label{eq.Gibbsu}
\G^u(g)=\{\mu\in \cM_{\inv}(g): h_\mu(g,\cF^u_g)\geq \int \log(\det(Tg\mid_{E^u(x)}))d\mu(x)\};
\end{equation*}

\item[(A2)] \begin{equation*}
\G^{cu}(g)=\{\mu\in \cM_{\inv}(g): h_\mu(g)\geq \int \log(\det(Tg\mid_{E^{cu}(x)}))d\mu(x)\}
\end{equation*}
where $E^{cu}=E^c\oplus E^u$.
\end{itemize}
\end{definition}

We denote by
$$\G(g)=\G^{u}(g)\cap \G^{cu}(g).$$

\begin{remark}\label{rk.Gu}
For $g\in C^{1+}$, by Ledrappier~\cite{L84}, $\G^u(g)=\Gibbs^u(g)$.
\end{remark}

We first observe that the spaces are non-empty.

\begin{proposition}\label{p.physical}
For any $g\in \cU$,
there is a full volume subset $\Gamma_g$ such that for any $x\in\Gamma$, any Cesaro limit of the
sequence $\frac{1}{n}\sum_{i=0}^{n-1}\delta_{g^i(x)}$ belongs to $\G(g)$.

\end{proposition}
\begin{proof}
By \cite{CCE}, for $x$ belonging to a full volume subset, any limit of the sequence $\frac{1}{n}\sum_{i=0}^{n-1}\delta_{g^i(x)}$
belongs to $G^{cu}$. Moreover, by \cite{CYZ,HYY}, for $x$ belonging to a full volume subset, any limit of the sequence
$\frac{1}{n}\sum_{i=0}^{n-1}\delta_{g^i(x)}$ belongs to $G^{u}$. We conclude the proof by taking the intersection
of the two full volume subsets.
\end{proof}

\begin{remark}\label{rk.uniqueG}
As a direct consequence, for $g\in \cU$, if $G(g)$ consists of a unique measure, then this measure is automatically a physical measure whose basin has full volume.
\end{remark}

\begin{lemma}\label{l.Gcompact}
$\G(g)$ is compact for any $g\in \cU$.
\end{lemma}
\begin{proof}
Assume there is a sequence of invariant measures $\mu_n\in \G(g)$ and they converge to $\mu$ in weak-* topology. From the definition, we have for each $n$:
\begin{align*}
h_{\mu_n}(g,\cF^u_g)&\geq \int \log(\det(Tg\mid_{E^u(x)}))d\mu_n(x)\\
h_{\mu_n}(g)&\geq \int \log(\det(Tg\mid_{E^{cu}(x)}))d\mu_n(x).
\end{align*}

By Theorem~\ref{main.D}, we have
$$\limsup_{n\to \infty}h_{\mu_n}(g,\cF^u_g)\leq h_{\mu}(g,\cF^u_g).$$
Since $g$ is partially hyperbolic diffeomorphism with one dimensional center, it is always entropy expansive and the metric entropy varies
upper semi-continuously (see \cite{LVY}), thus
$$\limsup_{n\to \infty}h_{\mu_n}(g)\leq h_{\mu}(g).$$
As a consequence, we have
\begin{align*}
h_{\mu}(g,\cF^u_g)&\geq \int \log(\det(Tg\mid_{E^u(x)}))d\mu(x),\\
h_{\mu}(g,\cF^{cu}_g)&\geq \int \log(\det(Tg\mid_{E^{cu}(x)}))d\mu(x).
\end{align*}
This implies $\mu\in \G(g)$. The proof is complete.
\end{proof}

We need the following two properties for the set
 $\G(\cdot)$, whose proof will be given in the next two subsections.

\begin{proposition}\label{p.Gunique}
For every $C^{1+}$ diffeomorphism $g\in \cU$, $\G(g)$ coincides with the unique physical measure of $g$.
\end{proposition}
Now we consider the map $\cG$ from the set of diffeomorphisms $\cU$ to the space of compact sets of probabilities:
$$\cG(g)=\G(g).$$
Recall that the space of compact subsets of a compact space is still  compact  under the Hausdorff topology.

\begin{proposition}\label{p.Gsemicontinuous}
The map $\cG(\cdot)$ is upper semi-continuous.
\end{proposition}

Let us finish the proof of Corollary~\ref{maincor.generic}.
\begin{proof}
Because the map $\cG(\cdot)$ is upper semi-continuous, it is continuous on a residual subset, i.e., there is
a $C^1$ residual subset of diffeomorphisms $\cR\subset \cU$ which are the continuous points of the map $\cG$.

Because $C^{1+}$ diffeomorphisms are $C^1$ dense among $\cU$, by Proposition~\ref{p.Gunique}, there is a dense subset of diffeomorphisms
in $\cU$ whose image under $\cG$ has a unique element. Then, be the continuity, for every diffeomorphism $g\in \cR$,
$\G(g)=\cG(g)$ consists of a unique measure. And by Remark~\ref{rk.uniqueG}, this measure is the unique physical measure of $g\in\cR$ and
whose basin has full volume.
\end{proof}
\subsection{Proof of Proposition~\ref{p.Gunique}}
By Theorem~\ref{main.C}, we denote by $\mu_g$ the unique physical physical measure of the $C^{1+}$ diffeomorphism $g\in\cU$.
Let $\mu$ be any invariant measure of $g$ belonging to $\G(g)$, we are going to show $\mu=\mu_g$.

Because $\mu\in \G^{u}(g)$, $h_{\mu}(g,\cF^u_g)\geq \int \log(\det(Tg\mid_{E^u(x)}))d\mu(x)$. By Remark~\ref{rk.Gu},
$\mu\in \Gibb^u(g)$.

By (1) of Proposition~\ref{p.Gibbsustates}, the ergodic decomposition of $\mu$ are still Gibbs $u$-states of $g$. Since $g$
has mostly expanding center, the center exponent of every ergodic Gibbs $u$-state of $g$ is always positive, thus the center exponents of
$\mu$ almost every point is positive. Take any Gibbs $u$-state $\tilde{\mu}$ which belongs to the ergodic decomposition of $\mu$, then by Ruelle inequality,
\begin{equation}\label{eq.Ruelle}
h_{\tilde{\mu}}(g)\leq \int \sum_{\lambda_i(x)>0}\lambda_i(x)d\tilde{\mu}(x)=\int \log(\det(Tg\mid_{E^{cu}(x)}))d\tilde{\mu}(x)\leq 0,\\
\end{equation}
$$\text{ thus}\;\;\; h_{\mu}(g)\leq \int \log(\det(Tg\mid_{E^{cu}(x)}))d\mu(x).$$

From the definition, $\mu\in \G^{cu}(g)$ means
$$h_{\mu}(g)\geq \int \log(\det(Tg\mid_{E^{cu}(x)}))d\mu(x).$$
Hence we have
$$h_{\mu}(g)= \int \log(\det(Tg\mid_{E^{cu}(x)}))d\mu(x).$$
Because the entropy function is affine, by \eqref{eq.Ruelle}, we have the equality
$$h_{\tilde{\mu}}(g)= \int \sum_{\lambda_i(x)>0}\lambda_i(x)d\tilde{\mu}(x).$$

By Ledrappier \cite{L84}, each ergodic decomposition $\tilde{\mu}$ of $\mu$ is a Gibbs $cu$ state, i.e., its
disintegration along its Pesin unstable manifold (with dimension $\dim(E^{cu})$) is absolutely continuous with respect  to the Lebesgue measure on leaves.
This implies that the regular point of $\tilde{\mu}$ has full Lebesgue measure on some Pesin unstable manifold which has
dimension $\dim(E^{cu})$. Because the basin is $s$-saturated, and the stable foliation is absolutely continuous, the
basin of $\tilde{\mu}$ has positive Lebesgue measure. Thus $\tilde{\mu}$ is a physical measure of $g$, which must coincide with the
unique physical measure of $g$. This  implies that $\mu$ coincides with $\mu_g$. The
proof is complete.

\subsection{Proof of Proposition~\ref{p.Gsemicontinuous}}
The proof is similar to the proof of Lemma~\ref{l.Gcompact}. We need to show that
for any diffeomorphism $g_n\in \cU$, $g_n\overset{C^1}{\longrightarrow} g$, we have
$$\limsup_{n\to \infty} \G(g_n)\subset \G(g),$$
it is equivalent to show that for any probabilities $\mu_n\in \G(g_n)$ converging to $\mu$ in the weak-* topology,
we have $\mu\in \G(g)$.

From the definition of space $\G(\cdot)$, we have for each $n$:
\begin{align*}
h_{\mu_n}(g_n,\cF^u_{g_n})&\geq \int \log(\det(Tg_n\mid_{E_{g_n}^u(x)}))d\mu_n(x),\\
h_{\mu_n}(g_n,\cF^{cu}_{g_n})&\geq \int \log(\det(Tg_n\mid_{E^{cu}_{g_n}(x)}))d\mu_n(x).
\end{align*}

By Theorem~\ref{main.D}, we have
$$\limsup_{n\to \infty}h_{\mu_n}(g,\cF^u_g)\leq h_{\mu}(g,\cF^u_g).$$
And because $\cU$ consists partially hyperbolic diffeomorphisms with 1 dimensional center, they are uniformly entropy expansive,
the metric entropy varies upper semi-continuously with respect to diffeomorphisms and measures (see \cite{LVY}), thus
$$\limsup_{n\to \infty}h_{\mu_n}(g_n)\leq h_{\mu}(g).$$

On the other hand, for partially hyperbolic diffeomoprhisms, the invariant bundles vary continuously with respect to the diffeomorphisms. This gives
\begin{align*}
\lim_{n\to \infty} \int \log(\det(Tg_n\mid_{E^u_{g_n}(x)}))d\mu_n(x) &= \int \log(\det(Tg\mid_{E^u(x)}))d\mu(x),\\
\lim_{n\to \infty} \int \log(\det(Tg_n\mid_{E^{cu}_{g_n}(x)}))d\mu_n(x) &= \int \log(\det(Tg\mid_{E^{cu}(x)}))d\mu(x).
\end{align*}

By the previous discussion, we have
\begin{align*}
h_{\mu}(g,\cF^u_g)&\geq \int \log(\det(Tg\mid_{E^u(x)}))d\mu(x),\\
h_{\mu}(g,\cF^{cu}_g)&\geq \int \log(\det(Tg\mid_{E^{cu}(x)}))d\mu(x).
\end{align*}
This implies $\mu\in \G(g)$. The proof is complete.

\section{Robustly transitive diffeomorphisms\label{s.transitive}}
Throughout this section, we assume $f$ to be a $C^{1+}$ volume preserving, accessible partially hyperbolic diffeomorphism with one
dimensional center, and the center exponent is non-vanishing. By \cite{BW}, $f$ is ergodic. For simplicity, we may assume the center exponent is positive:
$\lambda^c_{\Leb}(f)=\int \log(\det(Tf\mid_{E^u(x)}))d\Leb(x)>0$.

To prove that $f$ is robustly transitive, we are going to show that there is a $C^1$ neighborhood $\cU$ of $f$, for any diffeomorphism $g\in \cU$
and any open sets $U,V\subset M$, there are $n,m>0$ such that $g^n(U)\cap g^{-m}(V)\neq \emptyset$.

In the proof we use the invariant probability spaces $\G^{u}(\cdot), \G^{cu}(\cdot)$ and $\G(\cdot)$ for diffeomorphisms belong to $\cU$
and $\cU^{-1}=\{g^{-1}: g\in \cU\}$. Let us recall their definitions:
\begin{itemize}
\item[(A1)] \begin{equation}\label{eq.Gu}
\G^u(g)=\{\mu\in \cM_{\inv}(g): h_\mu(g,\cF^u_g)\geq \int \log(\det(Tf\mid_{E^u(x)}))d\mu(x)\}.
\end{equation}

\item[(A2)] \begin{equation}
\G^{cu}(g)=\{\mu\in \cM_{\inv}(g): \int \log(\det(Tf\mid_{E^{cu}(x)}))d\mu(x)\}.
\end{equation}

\end{itemize}
And
$$\G(g)=\G^{u}(g)\cap \G^{cu}(g).$$

By Lemma~\ref{l.singleme}, $f$ has mostly expanding center, that is, there is $b>0$ such that for any $\mu\in \Gibb^u(f)$,
$\lambda^c_\mu(f)=\int \log \|Tf\mid_{E^c(x)}\|d\mu(x)>b$.

By Proposition~\ref{p.Gunique}, $\G(f)=\{\Leb\}$. Furthermore, the function
$\cG: g\in \cU \to \G(g)$ is upper semi-continuous. Hence, when taking $\cU$ sufficiently small, for $g\in \cU$, $\G(g)$ is close to $\G(f)=\{\Leb\}$,
thus for any $\mu\in \G(g)$, $\lambda^c_{\mu}(g)=\int \log(\det(Tg\mid_{E^u(x)}))d\mu(x)>b$.

By Proposition~\ref{p.physical}, for any $g\in \cU$, there is a full volume subset $\Gamma_g$ such that for any $x\in\Gamma_g$, any Cesaro limit of the
sequence $\frac{1}{n}\sum_{i=0}^{n-1}\delta_{g^i(x)}$ belongs to $\G(g)$. Because the center bundle $E^c$ is 1 dimensional, this means that for any $x\in \Gamma_g$,
$$\liminf_{n\to \infty}\frac{1}{n}\sum_{i=0}^{n-1}\log \|Tf\mid_{E^c_g(g^i(x))}\|=\int \log \|Tg\mid_{E^c_g(x)}\| d\frac{1}{n}\sum_{i=0}^{n-1}\delta_{g^i(x)}>b>0.$$

Define $H(b/2)$ to be the set of $b/2$-hyperbolic times for $x$, that is, the set of times $m\geq 1$ such that
\begin{equation}\label{eq.hyperbolictime}
\frac{1}{k}\sum_{m-k}^{m-1} \log \|Tg\mid_{E^c_g(g^i(x))}\|\geq b/2 \text{ for all $1\leq k \leq m$}.
\end{equation}
By the Pliss Lemma (see \cite{ABV00}), there exists $n_1\geq 1$ and $\delta_1>0$ such that
$$\#(H(b/2)\cap [1,n))\geq n\delta_1 \text{ for all $n\geq n_1$ }.$$
By Lemma~\ref{l.hyperbolictime} and Remark~\ref{r.uniformunstable}, there is $\delta_1>0$
which only depends on $\cU$ and $b/2$, such that for any $x\in \Gamma_g$, and disk
$D$ tangent to the center-unstable cone field, $x\in D$, $n\in H(b/2)$,
$$d_{D}(x,y)\leq e^{-nb/2}d_{f^n(D)}(f^n(x),f^n(y)),$$
for any $y\in D$ with $d_{f^n(D)}(f^n(x),f^n(y))\leq \delta_1$.
In particular, for $x\in U\cap \Gamma_g$ and $D\subset U$, we can show that
$f^n(D)$ contains a subset $D_n$ with unform diameter $\delta_1$ for $n\in H(b/2)$ sufficient large.
Then $\cup_{z\in D_n}\cF^s(z)$ contains an open ball with radius $\delta_1/2$, which we denote by $B_g$.

In the following, we are going to show that there is $m>0$ such that $g^{-m}(V)\cap D_n\neq \emptyset$.

Because the stable foliation is expanding under the negative iteration of $g$, it suffices to show that for
some point $z\in V$, $g^{-m}(z)\in B_g$ for infinitely many $m$. Indeed we are going to show that such points
have full volume.

\begin{lemma}\label{l.Ginverse}
There is $p_0>0$ such that
any $\mu\in \G(f^{-1})$ can be written as
$$\mu=p\Leb+(1-p)\mu_{+},$$
where $\mu_+\in \G^{u}(f^{-})$ is a Gibbs $u$-state of $f^{-1}$ whose ergodic decomposition consists of ergodic Gibbs
$u$-states of $f^{-1}$ with positive center exponent, and $p\geq p_0$.
\end{lemma}
We leave the proof of this lemma to the end, and continue the proof of Theorem~\ref{main.transitive}.

We cover the ambient manifold $M$ with finitely many balls $\cB=\{B_1,\cdots, B_k\}$ each with fix radius $\delta>0$
such that every ball on the manifold with radius $\delta_1/2$ should contain some ball in $\cB$.
Denote by $b_0=\min\{\Leb(B_i)\}_{1\leq i\leq k}$, then for any measure $\mu\in \G(f^{-1})$ and any $1\leq i \leq k$,
$$\mu(B_i)> p_0 b_0.$$
There is a neighborhood $\cV$ of probabilities (not necessarily invariant for $f$) of $\G(f^{-1})$ such that for
any $\nu\in \cV$ and any $1\leq i \leq k$,
$$\nu(B_i)> p_0 b_0.$$

By Proposition~\ref{p.Gsemicontinuous}, shrink $\cU$ if necessary, for any $g\in \cU$ we have $\G(g^{-1})\subset \cV$.
By Proposition~\ref{p.physical}, for Lebesgue almost every point $x\in \Gamma_{g^{-1}}$, any limit
of the sequence $\frac{1}{n}\sum_{i=0}^{n-1}\delta_{g^i(x)}$ belongs to the compact set $\G(g^{-1})$.
Thus there is $n_x>0$ such that for any $n>n_x$, $\frac{1}{n}\sum_{i=0}^{n-1}\delta_{g^i(x)}\in \cV$.
Hence for every $B_i$, $\frac{1}{n}\sum_{i=0}^{n-1}\delta_{g^{-i}(x)}(B_i)>p_0b_0>0$, which  implies the $g^{-1}$-positive orbit
of $x\in \Gamma_{g^{-1}}$ intersects every $B_i$, and thus intersect $B_g$. We finish the proof of Theorem~\ref{main.transitive}.

It remains to prove Lemma~\ref{l.Ginverse}. The key idea here is to consider  the following functional on the space of invariant measures:
$$P(\mu) = h_{(\mu)}(f^{-1})- \int \log(\det(T(f^{-1})\mid_{E_{f^{-1}}^{cu}(x)}))d(\mu).$$
Then we have  $P(\Leb)>0$, and $P(\mu)<-a<0$ for any ergodic Gibbs $u$-state $\mu$ of $f^{-1}$ with positive center exponent.

\begin{proof}
Since $f^{-1}$ is ergodic, and the Lebesgue measure has negative center exponent, by (4) of Proposition~\ref{p.Gibbsustates}, it is the
unique ergodic Gibbs $u$-state of $f^{-1}$ with negative center exponent. Moreover, as shown in Lemma~\ref{l.singleme}, there is
no ergodic Gibbs $u$-state of $f^{-1}$ with vanishing center exponent.

By Pesin formula and Proposition~\ref{p.vanishingcenterexponents}, $$h_{\Leb}(f^{-1})=h_{\Leb}(f^{-1},\cF^u_{f^{-1}})=\log(\det(T(f^{-1})\mid_{E_{f^{-1}}^u(x)}))d\Leb(x).$$
Thus
\begin{eqnarray}\label{eq.positive}
   &\; & h_{\Leb}(f^{-1})- \int \log(\det(T(f^{-1})\mid_{E_{f^{-1}}^{cu}(x)}))d\Leb(x) \\
   &=& h_{\Leb}(f^{-1})- \int \log(\det(T(f^{-1})\mid_{E_{f^{-1}}^{u}(x)}))d\mu(x)-\lambda^c_{\Leb}(f^{-1}) \\
   &=& -\lambda^c_{\Leb}(f^{-1})>0.
\end{eqnarray}

For any ergodic Gibbs $u$-state $\mu$ of $f^{-1}$ with positive center exponent, By Ruelle's inequality we must have
\begin{equation}\label{eq.negative}
h_{\mu}(f^{-1})- \int \log(\det(T(f^{-1})\mid_{E_{f^{-1}}^{cu}(x)}))d\mu(x)<0.
\end{equation}
Otherwise, by Ledrappier \cite{L84}, the equality implies that $\mu$ is a Gibbs $cu$ state, i.e., its
disintegration along its Pesin unstable manifold (with dimension $\dim(E^{cu})$) equals to
the Lebesgue measure on leaves. This implies that the regular point of $\tilde{\mu}$ has
full Lebesgue measure of some Pesin unstable manifold which has dimension $\dim(E^{cu})$.
Because the basin is $s$-saturated, and the stable foliation is absolutely continuous, the
basin of $\mu$ has positive Lebesgue measure. Thus $\mu$ is a physical measure of $g$, which
contradicts to the fact that $\Leb$ is the unique physical measure of $f$.

From the definition of $\G^{cu}$, it remains for us to show that there is $c>0$ such that any ergodic Gibbs $u$-state $\mu$ of $f^{-1}$ with positive center exponent,
\begin{equation}\label{eq.uniformnegative}
h_{\mu}(f^{-1})- \int \log(\det(T(f^{-1})\mid_{E^{cu}_{f^{-1}}(x)}))d\mu(x)<-c.
\end{equation}

Suppose by contradiction that there is a sequence of ergodic Gibbs $u$-state $\mu_n$ of $f^{-1}$ with positive center exponent, such that
$$\lim_{n\to \infty} h_{\mu_n}(f^{-1})- \int \log(\det(T(f^{-1})\mid_{E^{cu}_{f^{-1}}(x)}))d\mu_n(x)=0.$$
Replacing by a subsequence, we may assume that $\mu_n\to \mu$ in weak-* topology. By the u pper semi-continuity of
metric entropy, then we have
\begin{equation}\label{eq.limitnonnegative}
h_{\mu}(f^{-1})- \int \log(\det(T(f^{-1})\mid_{E^{cu}_{f^{-1}}(x)}))d\mu(x)\geq 0.
\end{equation}

We consider the ergodic decomposition of $\mu$, by (i) of Proposition~\ref{p.Gibbsustates}:
$$\mu=\int_{\cM_{erg}(f^{-1})} m d\tau(m)=\int_{\Gibb^u(f^{-1})} m d\tau(m).$$
Because the metric entropy function is an affine function, by \eqref{eq.positive},\eqref{eq.negative}
and \eqref{eq.limitnonnegative}, $\tau(\Leb)=d>0$.

Now take a $\Leb$-typical unstable plaque $D$, i.e., a  unstable plaque such that  Lebesgue almost
every point of $D$ is the regular point of $f^{-1}$. They admit  absolutely continuous Pesin stable manifolds with
dimension $\dim(E^s)+1$, whose union, denoted
by $\Gamma^s$, belongs to the basin of $\Leb$ for map $f^{-1}$.

Take a small neighborhood $U$ of $D$, then $\mu(U)\geq d \Leb(U)>0$. So for $n$ sufficiently large, $\mu_n(U)>0$.
Fix such a large $n$ and take a $\mu_n$-typical unstable disk $D^\prime$ which belongs to
$U$, such that Lebesgue almost every point of $D^\prime$ belongs to the basin of $\mu_n$. Then
$\Gamma^s$ intersects $D^\prime$ on a positive Lebesgue measure subset and the intersection belongs to
the basin of $\Leb$, we get a contradiction.

The proof is complete.
\end{proof}

\appendix
\section{\label{ss.smallboundary}Proof of Proposition~\ref{p.smallboundary}}

\begin{proof}[Proof of Proposition~\ref{p.smallboundary}:]
We need the following lemma which is modified from~\cite[Proposition 3.2]{LS82}.

\begin{lemma}\label{l.interval}

Let $\mu$ be a probability measure supported on $[0,R]$. Then for any $0<\lambda<\lambda^{\prime}<1$,
there is a full Lebesgue subset $I^*\subset (0,R)$ such that for every $r\in I^*$, there is
$D_r>0$ satisfying that for every $n\ge 0$:
$$\mu([r-\lambda^n,r+\lambda^n])\leq (D_r \lambda^{\prime})^n.$$

\end{lemma}

\begin{proof}[Proof of Lemma~\ref{l.interval}:]

For each $n$, denote by
$$J_n=\{r\in [0,R]; \mu([r-\lambda^n,r+\lambda^n])\geq (\lambda^\prime)^n\}.$$
Now consider a covering of $J_n$ by $\{(r_i-\lambda^n,r_i+\lambda^n); r_i\in J_n\}_{i=1,\dots,t}$ such
that every point of interval is covered by at most twice. By the definition of $J_n$, we have
$$2\geq  \sum_{i=1}^t \mu ((r_i-\lambda^n,r_i+\lambda^n))\geq t (\lambda^{\prime})^n.$$
Hence, $t\leq \frac{2}{(\lambda^{\prime})^n}$. Then
$$\Leb(J_n)\leq \sum_{i=1}^t \Leb((r_i-\lambda^n,r_i+\lambda^n)=t2\lambda^n \leq 4\frac{\lambda^n}{(\lambda^{\prime})^n},$$
which implies that $\sum \Leb(J_n)<\infty$.

Then for Lebesgue almost every $r\in (0,R)$, there is $n_r$ such that $r\notin J_n$
for any $n\geq n_r$. We can choose $D_r>1$ such that $\mu((r-\lambda^i,r+\lambda^i))<(D_r \lambda^{\prime})^i$
for $i=1,\dots, n_r$. The proof is complete.
\end{proof}

Let us continue the proof. Denote $\mu_0=\mu$ and $K=\sum \frac{1}{n^2}$, we consider a new probability measure $\nu=\frac{1}{K}(\sum \frac{1}{(n+1)^2}\nu_n)$ of $M$. For every $x\in M$, denote the Borel measure
$\mu_{x}$ on $[0,R]$ by
$$\mu_{x}((a,b))=\nu (\{y: a \leq d(x,y)\leq b\}) \text{ where $(a,b)\subset [0,R]$}.$$
If $\mu_{x}([0,R])=0$, we take $r_x=\frac{R}{2}$. Otherwise, applying Lemma~\ref{l.interval},
we may choose $\frac{R}{2}<r_{x}<R$ and $D_x$ such that
$$\mu_x(\{y: r_{x}-\lambda^n\leq d(x,y)<r_x+\lambda^n\})\leq D_x (\lambda^{\prime})^n,$$
which is equivalent to say that:
\begin{equation}
\nu(\{y: r_{x}-\lambda^n\leq d(x,y)<r_x+\lambda^n\})\leq D_x (\lambda^{\prime})^n.
\end{equation}
We can take finitely many points $\{x_1,\dots, x_t\}$ such that $B_{r_{x_i}}(x_i)$ covers $M$, and denote
by $D=\max_{\leq i\leq t}{D_{x_i}}$.
Then for the partition $\cA=\vee_{i=1}^t\{B_{r_{x_i}}(x_i), (B_{r_{x_i}}(x_i))^c\}$ and any $i\geq 0$, we have
$\diam(\cA)<R$ and
$$\nu(B_{\lambda^i}(\partial\cA))\leq \sum_{j=1}^t \nu(B_{\lambda^i}(\partial (B_{r_{x_j}}(x_j))))\leq tD(\lambda^\prime)^i.$$

Hence, for each $n\geq 0$, $\nu_n(B_{\lambda^i}(\partial\cA))\leq K(1+n)^2 \nu(B_{\lambda^i}(\partial\cA))
\leq K(1+n)^2tD(\lambda^{\prime})^i$. We conclude the proof by taking $C_n=K(1+n)^2tD$.
\end{proof}

\end{document}